\documentclass[10pt]{article}
\usepackage{amsmath}
\usepackage{amsfonts}
\usepackage{amssymb}
\usepackage{latexsym}
\usepackage{theorem}
\usepackage{fancyhdr}
\usepackage[sub,ovp]{psfragx}
\usepackage{overpic,color}
\usepackage{longtable}
\usepackage{graphicx}
\usepackage{subfigure}

\topmargin by -\headsep

\advance\hoffset by 0mm  
\advance\voffset by  8mm  

\newcommand{\ndg}[1]{| \kern -.25mm \|{#1}| \kern -.25mm \|}

\newcommand{\ltwo}[2]{\|{#1}\|_{#2}}

\newcommand{\el}{ \kappa \in \mathcal{T} }

\newcommand{\ud}{\mathrm{d}}

\newcommand{\dint}{\text{\rm int}}

\newcommand{\ha}{\frac{1}{2}}

\newcommand{\qed}{ \vspace{-0.5cm} \hfill $\Box$ }
\newcommand{\mbf}[1]{\mbox{\boldmath$\rm{#1}$}}

\newcommand{\mean}[1]{ \{#1\} }
\newcommand{\jump}[1]{  [#1]  }
\newcommand{\jumptwo}[1]{[\![#1]\!]}

\newcommand{\lift}[1]{\mathcal{L}({#1})}

\newcommand{\h}[2]{\bigg(\kern -.4mm\frac{h_{#1}}{2}\kern -.4mm\bigg)^{\kern -1.1mm{#2}}}

\newcommand{\bn}{\mbox{\boldmath$\rm{n}$}}

\DeclareMathOperator{\ess}{ess}

\newtheorem{theorem}{Theorem}[section]
\newtheorem{definition}[theorem]{Definition}

\newtheorem{remark}[theorem]{Remark}

\newtheorem{lemma}[theorem]{Lemma}

\newenvironment{proof}[1][Proof]{\begin{trivlist}
\item[\hskip \labelsep {\bfseries #1}]}{\end{trivlist}\qed}

\newcommand{\R}{\mathbb{R}}

\def\eint{\Gamma_{\text{\rm int}}}
\def\eske{\Gamma}

\def\dom{\Omega}
\def\udg{U}
\def\udgell{\tilde{u}_h}
\def\timestep{\lambda}

\def\initpdesol{u_0}
\def\pdesol{u}
\def\pdesolell{\tilde{u}}

\def\ellrec{\omega}

\def\pdesource{f}
\def\pdesourcetimeapprox{\tilde{f}}
\def\pdesourceell{\phi}

\def\dualpdesolell{z}

\def\ellest{\mathcal{E}}
\def\parerr{\rho}
\def\ellerr{\epsilon}
\def\ellrecsource{g}

\def\testfuncLinfty{\tilde{v}}
\def\testfuncLtwo{\bar{v}}

\def\elem{\kappa}
\def\edge{e}

\def\theend{\hspace{.1cm} \text{.}}

\newcommand{\del}[1]{\Delta{#1}}
\newcommand{\dels}[1]{\Delta^{2}{#1}}
\newcommand{\bdel}[1]{\Delta_{h}{#1}}
\newcommand{\nab}[1]{\nabla{#1}}
\newcommand{\nabn}[1]{\nabla{#1} \cdot \vect{n}}

\newcommand{\vect}[1]{\mathbf{#1}}
\newcommand{\wavy}[1]{\mathcal{#1}}

\newcommand{\norm}[2]{||#1||_{#2}}
\newcommand{\snorm}[2]{|#1|_{#2}}
\newcommand{\enorm}[1]{|||#1|||}
\newcommand{\roots}[3]{{\displaystyle {\displaystyle (#1)}^{\frac{#2}{#3}}}}

\newcommand{\sfrac}[2]{ {\displaystyle \frac{{\displaystyle #1}}{{\displaystyle #2}}} }

\newcommand{\be}{\begin{equation}}
\newcommand{\ee}[1]{\label{#1} \end{equation}}
\newcommand{\bea}{\begin{eqnarray}}
\newcommand{\eea}{\end{eqnarray}}
\newcommand{\bale}[2]{\begin{equation} \left#1 \begin{array}{#2}}
\newcommand{\eale}[2]{\end{array} \right#1 \label{#2} \end{equation} }

\newcommand{\bal}[2]{\left#1 \begin{array}{#2}}
\newcommand{\eal}[1]{\end{array} \right#1  }

\newcommand{\ints}[2]{{\displaystyle \int_{#1}{#2} \hspace{.02in} {\rm d}s}}
\newcommand{\intx}[2]{{\displaystyle \int_{#1}{#2} \hspace{.02in} {\rm d}x}}

\newcommand{\femspace}{S^r}
\newcommand{\femspacen}[1]{S^{r,#1}}

\def\themesh{\wavy{T}}

\newcommand{\elsum}[1]{{\displaystyle \sum_{\elem \in \themesh}{#1}}}

\newcommand{\gap}{\hspace{.05in}}

\newcommand{\const}[1]{C_{#1}}

\def\constcont{C_{\text{\it{cont}}}}
\def\constcoer{C_{\text{\it{coer}}}}

\date{\today}

\title{Adaptive discontinuous Galerkin approximations\\ to fourth order parabolic problems}

\author{Emmanuil H.~Georgoulis$^*$ and
Juha M.~Virtanen\thanks{Department of Mathematics, University of Leicester,
University Road, Leicester LE1 7RH, United Kingdom. E-mail: {\tt Emmanuil.Georgoulis@le.ac.uk, jv77@le.ac.uk}}
}

\begin{document}
\maketitle

\begin{abstract}
An adaptive algorithm, based on residual type a posteriori indicators of errors measured in $L^\infty(L^2)$ and $L^2(L^2)$ norms, for a numerical scheme consisting of implicit Euler method in time and discontinuous Galerkin method in space for linear parabolic fourth order problems is presented. The a posteriori analysis is performed for convex domains in two and three space dimensions for local spatial polynomial degrees $r \ge 2$. The a posteriori estimates are then used within an adaptive algorithm, highlighting their relevance in practical computations, which results into substantial reduction of computational effort. 
\end{abstract}

\section{Introduction}
Fourth order parabolic equations and corresponding initial-boundary value problems appear in the modelling in areas as diverse as biology, phase-field modelling and image processing to name a few. In most cases of practical interest one has to resort to numerical methods for their solution, due to complex geometry and/or the presence of non-linearities.

During the last five decades, finite element methods (FEMs) have been widely used to numerically solve fourth order elliptic or parabolic problems; see, e.g., \cite{baker, ciarlet, brezzi.fortin:mixed,destuynder, hughes, brenner,moz.suli:nipg} and the references therein for earlier works. There are, generally speaking, three families of FEMs developed for fourth order problems: conforming, mixed and non-conforming. The classical conforming methods (see, e.g., \cite{ciarlet} and the references therein) require the construction of complicated elements with a number of degrees of freedom devoted to ensuring $C^1$-continuity across the element interfaces. This results into limitations in the applicability of conforming methods on general, possibly irregular, meshes \cite{stogner:06} and their non-trivial extensions to dimensions three (or higher). Mixed methods  (see, e.g., \cite{brezzi.fortin:mixed,destuynder} and the references therein), whereby the fourth order operator is first transformed into a system of second order operators are widely used in practice, but they require very careful treatment in the imposition of essential and natural boundary conditions. Non-conforming methods for fourth order problems were first presented by \cite{baker} and then further developed in \cite{hughes, brenner,moz.suli:nipg,MR2142199} and other works. The key idea in non-conforming methods is the use of penalties to ensure convergence into the natural energy space of the variational problem, despite finite element basis functions being either just continuous ($C^0$-interior penalty procedures; see, e.g., \cite{hughes, brenner}) or completely discontinuous (discontinuous Galerkin interior penalty procedures; see, e.g., \cite{baker,moz.suli:nipg,MR2142199,feng_kar}). 

Adaptive FEMs based on a posteriori error estimates has been an active field of research in recent years, especially for second order elliptic and parabolic problems. For the case of fourth order elliptic problems a posteriori error estimators and indicators have been developed, e.g., in \cite{MR1426319,verfurth,MR1902705,MR1836871,MR2291934,MR2373172,MR2670114, MR2752872,georgoulis_houston_virtanen}. A posteriori bounds and adaptive algorithms for parabolic fourth order problems are far less developed in the literature. For instance, the development of adaptive algorithms based on various types of a posteriori indicators for the Cahn-Hilliard fourth order parabolic problem can be found in \cite{larsson-mesforush,feng-wu:08,banas-nuernberg:09}. Error control for variational methods for fourth order parabolic equations has been predominantly focused to space-discrete mixed or conforming formulations. The recent work \cite{larsson-mesforush} deals with goal-oriented error estimation for the fully discrete Cahn-Hilliard problem. Therefore, the development of adaptive algorithms based on a posteriori estimators for fully discrete methods for fourth order parabolic problems is still largely an unexplored area.

Advances in a posteriori error analysis of fully discrete schemes with non-conforming spatial discretizations of second order parabolic problems have been recently presented \cite{ern-vohralik:10, georgoulis_lakkis_virtanen}. In \cite{georgoulis_lakkis_virtanen}, an adaptive algorithm based on the derived a posteriori estimates is also considered.
Local residual a posteriori error bounds for semi-discrete conforming and mixed spatial discretizations ffor the Cahn-Hilliard problem and the Hele-Shaw flow are presented in \cite{feng-wu:08}. Finally, a posteriori error estimates in an $L^2(H^2)$-type norm and adaptive algorithms for fully discrete schemes with discontinuous Galerkin methods for fourth order problems are proposed in \cite{thesis_juha}. The derivation of reliability bounds in \cite{thesis_juha} is based on the elliptic reconstruction framework of Makridakis and Nochetto \cite{MR2034895}; we also refer to \cite{lakkis-makridakis:06,georgoulis_lakkis_virtanen} for some relevant extensions.

This work is concerned with the derivation of a posteriori error estimates in weaker than $L^2(H^2)$-norms and their use within an adaptive algorithm for a class of discontinuous Galerkin interior penalty methods for a fully discrete approximation of the problem:
\begin{eqnarray} 
\pdesol_{t} +  \Delta^2 \pdesol &=& \pdesource \quad \;  \text{ in } \; \dom \times (0,T], \label{modelpde1}\\
\pdesol = \nabn{\pdesol} &= &0 \quad \; \, \text{ in } \; \partial\dom \times (0,T] \quad \text{ and } \label{modelpde2}\\
\pdesol &= &\initpdesol \quad \text{ in } \; \dom \times \{0\} \quad \text{} \label{modelpde3}
\end{eqnarray}
with $\dom \subset \mathbb{R}^d$, $d=2,3$ a convex polygonal domain with boundary $\partial\dom$.  More specifically, we derive a posteriori error estimators for the error measured in $L^\infty(L^2)$ and $L^2(L^2)$ norms for a numerical scheme consisting of discontinuous Galerkin method in space and simple backward-Euler time-stepping for the problem  (\ref{modelpde1}) - (\ref{modelpde3}). The a posteriori analysis is performed for convex domains (as is usual for these norms) in two and three space dimensions for local spatial polynomial degrees $r \ge 2$. To enable the optimality of the a posteriori estimators in the $L^\infty(L^2)$ and $L^2(L^2)$ norms, the elliptic reconstruction framework is employed. Moreover, the $L^2(L^2)$-norm analysis employs a special test function construction inspired from the a priori analysis of FEMs for wave problems in \cite{baker_wave}. Somewhat surprisingly, the use of this special testing, in conjunction with the elliptic reconstruction, results into the derivation of $L^2(L^2)$-norm a posteriori estimators via a standard energy argument. The efficiency of the a posteriori estimators is assessed numerically. The reliability bounds are used within two variants of a space-time adaptive algorithm. The adaptive algorithm is able to achieve the same error reduction with far fewer degrees of freedom compared to uniform meshes, thereby highlighting the relevance of the derived a posteriori estimates in practical computations. The simple model problem (\ref{modelpde1}) - (\ref{modelpde3}) appears to be sufficient in highlighting some of the challenges in the error estimation and adaptivity of finite element methods for more complex fourth order parabolic problems.  It appears that the derived a posteriori bounds and the respective adaptive algorithms can be modified in a straightforward fashion to include the original dG method of Baker \cite{baker} and $C^0$-interior penalty methods \cite{hughes, brenner}. 

The remaining of this work is organised as follows. In Section \ref{notaandprem}, notation is introduced and some standard results needed in the subsequent analysis are recalled. The discontinuous Galerkin (dG) method for the biharmonic problem, along with the derivation of posteriori error bounds for the dG approximation of the biharmonic problem in $L^2$-norm are derived in Section \ref{section:DGmethod}. The respective fully discrete scheme for the parabolic model problem (\ref{modelpde1}) - (\ref{modelpde3}) is given in Section \ref{section:nummethod}, while Section \ref{apost_parabolic} contains the derivation of residual type a posteriori estimates of errors in $L^\infty(L^2)$ and $L^2(L^2)$ norms for the fully discrete scheme.  The efficiency and reliability of the a posteriori estimators is tested on a range of uniform meshes in Section \ref{numerics}. The adaptive algorithm utilizing the a posteriori estimates in a series of numerical experiments are also presented in Section \ref{numerics}. Some concluding remarks regarding the results and possible extensions are given in Section \ref{finalremark}.

\section{Notation and preliminaries}\label{notaandprem}

The standard Hilbertian Lebesgue space is denoted by $L^2(\omega)$, for a domain $\omega\subset\mathbb{R}^d$, $(d=2,3)$,
with corresponding inner product $\langle\cdot,\cdot,\rangle_{\omega}$ and norm $\|\cdot\|_{\omega}$;
when $\omega=\Omega$, we shall drop the subscript writing $\langle\cdot,\cdot,\rangle$ and $\|\cdot\|$, respectively.
We also denote by $H^s(\omega)$, the standard Hilbertian Sobolev
space of index $s\ge 0$ of real-valued functions defined on
$\omega\subset\mathbb{R}^d$, along with the corresponding norm and seminorm
$\|\cdot\|_{s,\omega}$ and $|\cdot|_{s,\omega}$, respectively. For $1\le p\le +\infty$, we also define the spaces $L^p(0,T,H^s(\omega))$, consisting of all measurable functions $v: [0,T]\to H^s(\omega)$,
for which
\begin{equation}
  \begin{gathered}
    \|v\|_{L^p(0,T;H^s(\omega))}
    :=\Big(\int_0^T \|v(t)\|_{s,\omega}^p\ud t\Big)^{1/p}<+\infty,
    \quad \text{for}\quad 1\le p< +\infty, \\
    \|v\|_{L^{\infty}(0,T;X)}:={\ess \sup}_{0\le t\le
      T}\|v(t)\|_{s,\omega}<+\infty,\quad \text{for}\quad  p = +\infty.
  \end{gathered}
\end{equation}

Let $\themesh$ be a subdivision of $\dom$ into disjoint elements $\el$.
The subdivision $\themesh$ is assumed to be shape-regular (see, e.g.,
p.124 in \cite{ciarlet}) and is constructed via smooth mappings
$F_{\elem}:\hat{\elem}\to\elem$ with uniformly bounded Jacobian throughout the mesh family considered, where $\hat{\elem}$ is the reference element.  
The above mappings are assumed to be constructed so as to ensure $\bar{\dom}=\cup_{\el}\elem$ and that
the elemental edges are straight segments (i.e., lines or planes). Note that we also use the expression edge to mean side when $d=3$.

The broken Laplacian, $\bdel{u}$, is defined element-wise by
$(\bdel{u})|_{\elem}:=\del({u}|_{\elem})$ for all  $\elem \in \themesh$.

For a nonnegative integer $r$, we denote by $\mathcal{P}_r(\hat{\elem})$, the set of all polynomials of
total degree at most $r$, if $\hat{\elem}$ is the reference simplex, or of degree at most $r$ in each variable, if $\hat{\elem}$ is the reference hypercube.
We consider the finite element space
\begin{equation}\label{eq:FEM-spc}
  \femspace  :=\{v\in L^2(\dom):v|_{\elem}\circ F_{\elem}
    \in\mathcal{P}_{r}(\hat{\elem}),\,\el\}.
\end{equation}

By $\Gamma$ we denote the union of all $(d-1)$-dimensional element
edges associated with the subdivision $\themesh$, including the boundary. Further, we decompose $\Gamma$ into
two disjoint subsets $\Gamma=\partial\Omega\cup\Gamma_{\dint}$,
where $\Gamma_{\dint}:=\Gamma\backslash\partial\Omega$.

For two (generic) elements $\elem^+,\elem^-\in\themesh$ sharing an edge $e=\elem^+\cap\elem^-$, we define the outward normal unit vectors $\mbf{n}^+$ and $\mbf{n}^-$ on $e$ corresponding
to $\partial\elem^+$ and $\partial\elem^-$, respectively. For functions $v:\dom\to\mathbb{R}$ and
$\mbf{q}:\dom\to\mathbb{R}^d$, that may be discontinuous across $\Gamma$,
we define the following quantities.
For $v^+:=v|_{e\subset\partial\elem^+}$, $v^-:=v|_{e\subset\partial\elem^-}$,
$\mbf{q}^+:=\mbf{q}|_{e\subset\partial\elem^+}$, and $\mbf{q}^-:=\mbf{q}|_{e\subset\partial\elem^-}$, we set
\[ \mean{v}:=\ha(v^+ + v^-),\ \mean{\mbf{q}}:=\ha(\mbf{q}^+ + \mbf{q}^-),
\qquad \jumptwo{v}:=q^+\mbf{n}^++q^-\mbf{n}^-,\ \jump{\mbf{q}}:=\mbf{q}^+\cdot \mbf{n}^++\mbf{q}^-\cdot \mbf{n}^-;\]
if $e\in \partial\elem\cap\partial\Omega$, these definitions are modified to $\mean{v}:=v^+$, $\mean{\mbf{q}}:=\mbf{q}^+$,
 $\jumptwo{v}:=v^+\mbf{n}$, $\jump{\mbf{q}}:=\mbf{\mbf{q}}^+\cdot \mbf{n}$.
With the above definitions, it is easy to verify the identity
\begin{equation}\label{scalarformula}
\elsum{\ints{\partial \elem}{v\,\mbf{q} \cdot \vect{n} } } =
\ints{\Gamma}{ \jumptwo{v}\cdot\mean{\mbf{q}}}  + \ints{\eint}{ \mean{v}\jump{\mbf{q}}},
\end{equation}
with $\mbf{n}$ denoting the outward normal unit vector on $\partial\elem$, corresponding to $\elem$.

We define the element size $h_{\elem}:=(\mu_{d}(\elem))^{1/d}$, where $\mu_{d}$ is the $d$-dimensional Lebesgue measure; we collect the element sizes into the 
into the element-wise constant function ${\bf h}:\dom\to\mathbb{R}$, with
${\bf h}|_{\elem}=h_{\elem}$, $\el$ and $\mbf{h}=\mean{{\mbf h}}$ on $\Gamma$. Also, for two (generic) elements  $\elem^+$, $\elem^-$ sharing an edge $e:=\partial\elem^+\cap\partial\elem^-\subset\Gamma_{\dint}$, we
define $h_{\edge}:=\mu_{d-1}(\edge)$.

As we shall be dealing with mesh adaptive algorithms below, we assume that all sequences of meshes considered in this work are locally quasi-uniform, i.e.,
there exists constant $c\ge 1$,
independent of ${\bf h}$, such that, for any pair of
elements $\elem^+$ and $\elem^-$ in $\themesh$ which share an edge,
\begin{equation}\label{eq:bdd-var}
 c^{-1}\le h_{\elem^+}/h_{\elem^-}\le c.
\end{equation}

Finally we recall a series of some (standard) results used throughout this work; their proofs can be found, e.g., in \cite{ciarlet,dupont-scott:80,braess:01,brenner-scott:02,brenner-wang-zhao:02}. 
\begin{lemma}[approximation property] \label{lemma:approximationproperty}
Let $0 \le m \le r + 1$ and $\themesh$ be a subdivision of $\dom$, $\dom \subset \R^d$. Then there exists a constant $C_{\text{app}}$, independent of $h_{\kappa}$, such that for any $u \in H^m(\dom)$ and $\kappa \in \themesh$, there exists $p:C(\elem)\to\R$, with $p\circ F_{\elem} \in \mathcal{P}_r(\elem)$ and 
\begin{equation}\label{approximationproperty}
\snorm{u-p}{j,\kappa} \le C_{\text{app}} \ h^{m-j}_{\kappa} \ \snorm{u}{m,\kappa} \ , 0 \le j \le m .
\end{equation}
\end{lemma}
\begin{lemma}[inverse estimate] \label{lemma:inverseestimate}
There exists a constant $C_{\text{inv}}$, independent of $h_{\elem}$, such that 
\begin{equation}\label{inverseestimate}
\snorm{p}{j,\kappa} \le C_{\text{inv}} \ h^{i-j}_{\kappa} \ \snorm{p}{i,\kappa} \ , 0 \le i \le j \le 2,
\end{equation}
for all $p:C(\elem)\to\R$, with $p\circ F_{\elem} \in \mathcal{P}_r(\elem)$.
\end{lemma}
\begin{lemma}[trace inequality] \label{lemma:traceinequality}
 For every $u \in H^1(\kappa)$, with $\elem \in\themesh$, there exists a constant $C_{\text{tr}} > 0$ independent of $h_{\elem}$ such that
\begin{equation} \label{traceinequality}
\norm{u}{0,\partial \elem}^2 \le C_{\text{tr}} ( h^{-1}_{\elem} \norm{u}{0,\kappa}^2 + h_{\elem} \snorm{u}{1,\kappa}^2 ). 
\end{equation}
\end{lemma}
\begin{lemma}[Poincar\'e-Friedrichs inequality \cite{brenner-wang-zhao:02}] \label{lemma:pfinequality}
There exists a constant $C_{\text{pf}}$, independent of $h_{\kappa}$, such that for any $u\in L^2(\dom)$, with $u|_{\elem} \in H^2(\elem)$ for all $\elem\in\themesh$, we have
\begin{equation} \label{pfinequality}
\norm{u}{0,\dom}^2 + \snorm{u}{1,\dom}^2 \le C_{\text{pf}} \left( \snorm{u}{2,\dom}^2 +\norm{\mbf{h}^{-3/2}\jumptwo{u}}{0,\Gamma}^2 + \norm{\mbf{h}^{-1/2}\jump{\nabla u}}{0,\Gamma}^2  \right).
\end{equation}
\end{lemma}

\section{Discontinuous Galerkin method for the biharmonic problem} \label{section:DGmethod}

We consider the biharmonic equation
\begin{equation}\label{pde}
  \Delta^2 \pdesolell= \pdesourceell \quad\text{in } \dom,
\end{equation}
with homogeneous essential boundary conditions
\begin{equation}\label{modelbc}
\pdesolell=0\quad,\qquad
\nabla \pdesolell \cdot\bn=0\quad\text{on }\partial\dom ,
\end{equation}
where $\bn$ denotes the unit outward normal vector to $\partial\dom$ and $\pdesourceell \in L^2(\dom)$. Then the regularity of the problem implies that $\pdesolell \in H^4(\dom) \cap H^2_0(\dom)$ \cite{grisvard:elliptic}.

Upon defining the lifting operator $\mathcal{L}:\mathcal{S}:=\femspace+H_0^2(\dom)\to \femspace$  by
\be
 \int_{\dom} \mathcal{L}(\nu) \psi \,\ud x =  \int_{\Gamma}
                                  \Big(\jumptwo{\nu}\cdot\mean{\nabla \psi} -  \mean{\psi}\jump{\nabla \nu}\Big)\,\ud s\quad\forall \psi\in \femspace,
\ee{lifting}
 the \emph{(symmetric) interior penalty discontinuous Galerkin (dG) method} for  (\ref{pde}), (\ref{modelbc})  is given by:
\be
  \text{find}\ \udgell \in \femspace\ \text{such that}\quad B(\udgell,v_h)=l(v_h)\quad\forall v_h\in \femspace,
\ee{dgfem}
where the bilinear form $B:\mathcal{S}\times\mathcal{S}\to\mathbb{R}$ and the linear form $l:\mathcal{S}\to\mathbb{R}$ are given by
\begin{equation}\label{bilinear}
\begin{aligned}
B(w,v)   :=&   \int_{\dom} \Big(\Delta_h w \Delta_h v
                                       + \lift{w}\Delta_h v+\Delta_h w\lift{v}\Big)\,\ud x+ B_p(w,v)
\end{aligned}
\end{equation}
with
\[
B_p(w,v):=\int_{\Gamma}\Big( \sigma\jumptwo{w}\cdot\jumptwo{v}
                                         + \xi\jump{\nabla w}\jump{\nabla v} \Big)\,\ud s,
\]
and
\be
l(v): = \int_{\dom} \pdesourceell v \,\ud x,
\ee{linear}
respectively, for $w,v\in\mathcal{S}$. The piecewise constant discontinuity penalization parameters $\sigma,\xi: \eske \rightarrow \R$ are given by
\be
 \sigma|_{\edge}  =  \sigma_0 (\mbf{h}|_{\edge})^{-3}, \quad \xi|_{\edge}  =  \xi_0 (\mbf{h}|_{\edge})^{-1},
\ee{penaltyparametersdef}
respectively, where $\sigma_0 > 0$ and $\xi_0 > 0$. To guarantee the stability of the IPDG method defined
in (\ref{dgfem}), $\sigma_0$ and $\xi_0$ must be selected sufficiently large.

Note that this formulation is inconsistent for trial and test functions belonging to the solution space $\mathcal{S}$.
However, when $w,v\in\femspace$, in view of (\ref{lifting}), (\ref{bilinear}) gives
\begin{equation}\label{bilinear_original}
\begin{aligned}
B(w,v)   =&   \int_{\dom} \Delta_h w \Delta_h v\,\ud x+
\int_{\Gamma}\Big( \mean{\nabla\Delta w}\cdot\jumptwo{v}+\mean{\nabla\Delta v}\cdot\jumptwo{w}\\
                                       &-\mean{\Delta w}\jump{\nabla v}-\mean{\Delta v}\jump{\nabla w}
                                        +\sigma\jumptwo{w}\cdot\jumptwo{v}
                                         + \xi\jump{\nabla w}\jump{\nabla v} \Big)\,\ud s;
\end{aligned}
\end{equation}
therefore, (\ref{dgfem}) coincides with the symmetric version interior penalty method presented in \cite{suli.moz:sipg}.
For the bilinear form $B(\cdot,\cdot)$ in (\ref{bilinear}) we have the continuity and coercivity with respect to the energy norm on $\mathcal{S}$ defined by
\bale{.}{rcl}
\enorm{w}&=&\roots{ \norm{\Delta_h w}{\Omega}^2 + \norm{\sqrt{\sigma} \jumptwo{w}}{\eske}^2 + \norm{\sqrt{\xi}\jump{\nab{w}}}{\eske}^2 }{1}{2}.
\eale{.}{energynorm}
\begin{lemma}[\cite{georgoulis_houston}] \label{contcoerlemma}
For sufficiently large $\sigma_0>0$ and $\xi_0 > 0$ there exist positive constants $\constcont$ and $\constcoer$, depending only on the mesh parameters such that
\be
|B(u,v)| \leq \constcont \enorm{u} \gap \enorm{v} \gap \forall u,v \in \mathcal{S} \quad \text{and}
\ee{continuitythm}
\be
B(u,u) \geq \constcoer \enorm{u}^2 \gap \forall u \in \mathcal{S} \theend
\ee{coercivitythm}
\end{lemma}

An a posteriori bound for the energy norm error of the dG method (\ref{dgfem}) for (\ref{pde}) -- (\ref{modelbc}) has been considered in \cite{georgoulis_houston_virtanen}. Now, we shall present an a posteriori bound for the $L^2$-norm error (cf. \cite{riviere-wheeler:03} for a corresponding result for the second order problem).

\begin{theorem}[$L^2$-a posteriori bounds for the elliptic problem]  \label{theorem:ltwo_apost_biharmonic}
Let $\pdesolell \in H^4(\dom)  \cap H^2_0(\dom)$ be the solution of (\ref{pde})--(\ref{modelbc}), $\udgell \in \femspace$
be dG approximation (\ref{dgfem}) associated with the mesh $\themesh$.
Then, there exists a positive constant $\const{(\ref{theorem:ltwo_apost_biharmonic})}$, independent of $\themesh$, $\mbf{h}$, $\pdesolell$ and $\udgell$, such that
\begin{equation}\label{ltwo_apost_bound_bound}
\ltwo{\pdesolell-\udgell}{} \le   \ellest \left(\themesh, \udgell, \pdesourceell \right),
\end{equation}
where 
\begin{equation}\label{ltwo_apost_estimator}
\begin{aligned}
\ellest \left(\themesh, \udgell, \pdesourceell \right) & := \const{(\ref{theorem:ltwo_apost_biharmonic})}\Big( \ltwo{\mbf{h}^{4 - \lambda/2}(\pdesourceell - \Delta_h^2 \udgell)}{}^2
+ \ltwo{\mbf{h}^{(7 - \lambda)/2}\jump{ \nabla \Delta \udgell}}{\Gamma_{\dint}}^2 +  \ltwo{\mbf{h}^{(5 - \lambda)/2}\jumptwo{ \Delta \udgell}}{\Gamma_{\dint}}^2 \\
& \phantom{ := } + \sum_{\edge \in \Gamma} \left( h_{\edge}^{3 - \lambda} \left( 1 + \xi_0^2 \right) {\ltwo{\jump{ \nabla \udgell}}{}}^2 + h_{\edge}^{1 - \lambda} \left( 1 + \sigma_0^2 \right) {\ltwo{\jumptwo{ \udgell}}{}}^2 \right)
\end{aligned}
\end{equation}
and $\lambda := 2 \, (2 - \min \{2, r - 1\})$.
\end{theorem}
\begin{proof}
The dual problem
\begin{equation}\label{dualpde}
  \Delta^2 \dualpdesolell= \pdesolell - \udgell=:\tilde{e} \quad\text{in } \dom,
\end{equation}
with homogeneous essential boundary conditions $\dualpdesolell=\nabla \pdesolell \cdot\bn=0 \;\text{on }\partial\dom $ clearly satisfies $\tilde{e} \in L^2(\dom)$ and, therefore, the following regularity estimate holds
\begin{equation}\label{dualpdereg}
  \ltwo{\dualpdesolell}{4,\dom} \le \const{\text{reg}} \ltwo{\tilde{e}}{}.
\end{equation}
Using (\ref{dualpde}), integrating by parts twice, applying (\ref{scalarformula}) and (\ref{lifting}) as well as the regularity of the dual solution, $\dualpdesolell \in H^4(\dom)$, we have
\begin{equation}\label{ltwo_apost_estimator_identity1}
\begin{aligned}
\ltwo{\tilde{e}}{}^2  =& \sum_{\elem \in \themesh} \int_{\elem} \, \Delta^2 \dualpdesolell \; \tilde{e} \, \ud x 
= \int_{\dom} \Delta_h \dualpdesolell  \Delta_h\tilde{e} \ud x  
 - \int_{\Gamma}  \jump{\nabla\tilde{e}} \mean{\Delta \dualpdesolell }  \ud s  +   \int_{\Gamma} \jumptwo{\tilde{e}} \cdot \mean{\nabla  \Delta \dualpdesolell} \ud s.
\end{aligned}
\end{equation}
By using the fact that $\pdesolell$ is a weak solution and integrating the term involving $\Delta_h \dualpdesolell  \Delta_h \udgell$ by parts, we arrive at,
\begin{equation}\label{ltwo_apost_estimator_intbypar2}
\begin{aligned}
\ltwo{\tilde{e}}{}^2   
= &\; B(\pdesolell, \dualpdesolell) - \int_{\dom} \Delta_h \dualpdesolell  \Delta_h \udgell \ud x  
 + \int_{\Gamma}  \jump{\nabla \udgell} \mean{\Delta \dualpdesolell }  \ud s  -   \int_{\Gamma} \jumptwo{\udgell} \cdot \mean{\nabla  \Delta \dualpdesolell} \ud s \\
= &\; l(\dualpdesolell) - \int_{\dom} \dualpdesolell  \Delta^2 \udgell \ud x  
+ \ints{\Gamma_{\dint}}{\mean{\dualpdesolell}\jump{\nabla\Delta_h \udgell}} 
- \ints{\Gamma_{\dint}}{\mean{\nabla \dualpdesolell}\cdot\jumptwo{\Delta_h{\udgell}}} \\
& \;+ \int_{\Gamma} \jump{\nabla \udgell} \mean{\Delta \dualpdesolell }  \ud s  -   \int_{\Gamma} \jumptwo{\udgell} \cdot \mean{\nabla  \Delta \dualpdesolell} \ud s.
\end{aligned}
\end{equation}

\noindent Using the standard orthogonal $L^2$-projection, $\Pi : \mathcal{S} \to \femspace$, of $\dualpdesolell$ , we can derive the following identity by integrating by parts and using (\ref{scalarformula}) and (\ref{lifting}) as follows,
\begin{equation}\label{ltwo_apost_estimator_conferror}
\begin{aligned}
& 0 = l( - \Pi \dualpdesolell) - B(\udgell, - \Pi \dualpdesolell) \\
=&\sum_{\elem \in \themesh} \intx{\elem}{\big((\pdesourceell-\dels{\udgell})(- \Pi \dualpdesolell) -\wavy{L}(\udgell)\Delta_h{(- \Pi \dualpdesolell)}\big)}
+ \ints{\Gamma_{\dint}}{\mean{- \Pi \dualpdesolell}\jump{\nabla\Delta_h \udgell}} \\
&- \ints{\Gamma_{\dint}}{\mean{\nabla(- \Pi \dualpdesolell)}\cdot\jumptwo{\Delta_h{\udgell}}}
  - \ints{\Gamma}{\big(\sigma\jumptwo{\udgell}\cdot\jumptwo{- \Pi \dualpdesolell}+\xi\jump{\nabla\udgell}\jump{\nabla(- \Pi \dualpdesolell)}\big)} .
\end{aligned}
\end{equation}
Using (\ref{lifting}) in (\ref{ltwo_apost_estimator_conferror}) and combining (\ref{ltwo_apost_estimator_intbypar2}) and (\ref{ltwo_apost_estimator_conferror}), we get
\begin{equation}\label{ltwo_apost_estimator_identity2}
\begin{aligned}
 \ltwo{\tilde{e}}{}^2  \!= &  \ltwo{\tilde{e}}{}^2 + l( - \Pi \dualpdesolell) - B(\udgell, - \Pi \dualpdesolell)  \\
 \!= & \! \int_{\dom}  (\pdesourceell - \Delta_h^2\udgell)  {(\dualpdesolell - \Pi \dualpdesolell)} \ud x
+ \! \int_{ \Gamma_{\dint}}\!\! \Big(\mean{\dualpdesolell - \Pi \dualpdesolell}  \jump{\nabla \Delta \udgell}  \ud s
- \jumptwo{\Delta  \udgell} \cdot \mean{\nabla (\dualpdesolell - \Pi \dualpdesolell)} \Big) \ud s  \\
& -\int_{\Gamma} \jumptwo{\udgell} \cdot \Big( \mean{\nabla \Delta (\dualpdesolell - \Pi \dualpdesolell) }  +\sigma_0 \mbf{h}^{-3} \jumptwo{\dualpdesolell - \Pi \dualpdesolell} \Big) \ud s\\
& +\int_{\Gamma} \jump{\nabla(\udgell)} \Big( \mean{ \Delta (\dualpdesolell - \Pi \dualpdesolell) }  +\xi_0 \mbf{h}^{-1} \jump{\nabla(\dualpdesolell - \Pi \dualpdesolell)} \Big) \ud s .
\end{aligned}
\end{equation}
The assertion then follows by applying Young's inequality, the trace inequality (\ref{traceinequality}) where appropriate, the approximation property (\ref{approximationproperty}) and the regularity of the dual problem on each of the terms on the right hand side of (\ref{ltwo_apost_estimator_identity2}).
\end{proof}

\begin{remark} \label{remark:smoothsubspace}
If a smooth $C^1$ subspace of the finite element space exists, such as Argyris elements in two dimensions, or corresponding constructions in three dimensions, it is possible to establish an a posteriori $L^2$ bound without dependence on penalty parameters; indeed, these terms would vanish from (\ref{ltwo_apost_estimator_identity2}) if the projection, $\Pi$, could be defined onto the smooth subspace of $\femspace$.
\end{remark}

\begin{remark} \label{remark:suboptimality}
 It is interesting to note that the a posteriori error bound of (\ref{ltwo_apost_bound_bound}) reflects the suboptimal $L^2$-norm error convergence of the dG method when quadratic polynomials are applied. Similar behaviour is observed theoretically and numerically in \cite{georgoulis_houston} and in \cite{suli.moz:sipg} in the context of the a priori error analysis of the same method.
\end{remark}

\section{DG method for the parabolic problem} \label{section:nummethod}

Throughout the remaining of this work, we shall denote by $\pdesol$ the weak solution of the problem (\ref{modelpde1})--(\ref{modelpde3}) in variational form:
find $\pdesol \in H^1(0,T; H^4(\dom) \cap H^2_0(\dom) )$  such that  
\begin{equation} \label{varmodelpde}
  \begin{aligned}
\langle \pdesol_t,  \phi \rangle + B(\pdesol,\phi)&=\langle \pdesource,\phi\rangle \quad \forall \phi \in H^2_0(\dom),  \\
\pdesol &= \initpdesol \in L^2(\dom) \quad \text{ in } \; \dom \times \{0\}.\quad\quad
  \end{aligned}
\end{equation}
We consider a subdivision of the time interval $(0,T]$ to be the family of intervals $\{(t^{n-1},t^{n}]$ ; $n=1,\dots, N$, with $t^0=0$, $t^{n-1} \le t^{n}$ and $t^N=T\}$ , with local time-step $\timestep_n:=t^{n}-t^{n-1}$. Associated with this time-subdivision, let $\themesh_{n}$, $n=0,\dots, N$, be a sequence of meshes which are
assumed to be \emph{compatible}, in the sense that for any two
consecutive meshes $\themesh_{n-1}$ and $\themesh_{n}$, $\themesh_{n}$ can be obtained
from $\themesh_{n-1}$ by locally coarsening some of its elements and then locally refining some (possibly other) elements.
The finite element space corresponding to $\themesh_{n}$ will be denoted by $\femspacen{n}$ and the respective dG bilinear form by $B^n(\cdot,\cdot)$. The backward Euler-dG method for approximating (\ref{varmodelpde}) is then given by: for each $n=1,\dots,N$, find
\begin{equation}\label{eqn:fully-discrete-Euler-pure}
  \begin{split}
    \ \udg^{n}\in \femspacen{n}\text{ such that  }
    \langle \frac{U^{n}- U^{n-1}}{\timestep_{n}},V\rangle+B^n(\udg^{n},V)=\langle \pdesourcetimeapprox^{n},V\rangle
    \quad \forall V\in \femspacen{n},
  \end{split}
\end{equation}
where $\pdesourcetimeapprox^0(\cdot):=\pdesource(\cdot,0)$ and $\pdesourcetimeapprox^n(\cdot)$ for $n=1,\dots,N$ is a piecewise polynomial of degree $p$ in time $L^2$-projection in time of the source function $\pdesource$. In practice, it suffices to take $p=0$ to achieve a first-order-in-time convergent method. We also set $U^0:=\Pi^0 \initpdesol$, with $\Pi^0: L^2(\dom) \to \femspacen{0}$ is the orthogonal $L^2$-projection operator onto the finite element space $\femspacen{0}$.

\section{A posteriori bounds for the parabolic problem}\label{apost_parabolic}

We shall derive a posteriori error bounds for the backward Euler-dG method (\ref{eqn:fully-discrete-Euler-pure}) measured in $L^\infty(L^2)$- and $L^2(L^2)$-norms. To this end, we shall employ an energy argument (with carefully defined test functions) in conjunction with the elliptic reconstruction technique \cite{MR2034895,lakkis-makridakis:06,georgoulis_lakkis_virtanen}.

We begin by extending the sequence ${\{\udg^n\}}_{n=1,\dots, N}$ of numerical solutions into a continuous piecewise linear
function of time
\begin{equation} \label{eqn:fully-discrete-solution-extended}
  \udg (0)=\Pi^0 \initpdesol \quad \text{ and } \quad \udg(t):= \sfrac{t - t^{n-1}}{\lambda_n} \udg^n + \sfrac{t^{n} - t}{\lambda_n} \udg^{n-1}
\end{equation}
for $t \in (t_{n-1},t_{n}]$ and $n=1,\dots, N$. Further, the \emph{discrete elliptic operator} $A^n:\femspacen{n} \to \femspacen{n}$ is defined by
\begin{equation}\label{ell_rec2}
\text{for}\ \phi \in \femspacen{n},\quad \langle A^n \phi, \chi \rangle = B^n(\phi,\chi)\quad \forall \chi\in \femspacen{n}.
\end{equation}

We now give definitions of the estimators involved in the estimation of the parabolic part of the error. Estimators at time step $n$ are denoted by ${\tiny \infty,n}$ subscript and ${\tiny 2,n}$ subscript will be used for the cases of $L^\infty(L^2)$- and $L^2(L^2)$-bounds presented below, respectively.
\begin{definition}[estimators for the parabolic error] \label{def_estimators}
We define the \emph{coarsening} or \emph{mesh-change estimators} by
\begin{equation}\label{apost_parabolic_coarse_estimators}
\gamma_{\infty,n}  := \frac{1}{\timestep_n}  \ltwo{(I-\Pi^n)\udg^{n-1}}{}^2,\quad 
\gamma_{2,n}  := {\ltwo{(I-\Pi^n)\udg^{n-1}}{}}^2 + \sum_{i=1}^{n-1} {\ltwo{ (\Pi^i - \Pi^{i-1}) \udg^{i-1} }{}}^2,
\end{equation}
the \emph{time-error evolution estimators} by
\begin{equation}\label{apost_parabolic_time_estimators}
\eta_{\infty,n}  := \ltwo{\ellrecsource^n - \ellrecsource^{n-1}}{}^2 {\timestep_n}, \quad
\eta_{2,n}  :=  \ltwo{\ellrecsource^n - \ellrecsource^{n-1}}{}^2 \timestep_n^{2} + \sum_{i=1}^{n-1} \timestep_{i}^{2}  \ltwo{ \ellrecsource^i - \ellrecsource^{i-1}  }{}^2,
\end{equation}
$\ellrecsource^n := A^n \udg^n-\Pi^n \pdesourcetimeapprox^n +\pdesourcetimeapprox^n$; the \emph{data approximation error in time estimators} by
\begin{equation}\label{apost_parabolic_data_estimators}
\beta_{\infty,n}  := \int_{t^{n-1}}^{t^{n}} \ltwo{ \pdesourcetimeapprox^n - \pdesource}{}^2\ud t, \quad
\beta_{2,n}  := \timestep_n \int_{t^{n-1}}^{t^{n}} \ltwo{ \pdesourcetimeapprox^n - \pdesource}{}^2\ud t,
\end{equation}
and an additional space estimator given by
\begin{equation}\label{apost_elliptic_extra_linfty_estimator}
\tilde{\eta}_{\infty,n}  := \ellest\Big(\hat{\themesh}_n, \udg^n-\udg^{n-1},\ellrecsource^n - \ellrecsource^{n-1} \Big)^2,
\end{equation}
where $\hat{\themesh}_n := \themesh_{n} \cap \themesh_{n-1}$ {\it finest common coarsening} of $\themesh_{n}$ and $\themesh_{n-1}$ for each $n=1,\dots, N$.
\end{definition}

Using the notation above, we are ready to state the main result.

\begin{theorem}[a posteriori bound] \label{theorem:apost_parabolic}
Let $\pdesol \in L^2(0,T;H^4(\dom)  \cap H^2_0(\dom))$ be the solution of (\ref{varmodelpde}), $ \udg$
be the approximation obtained by the dG method (\ref{eqn:fully-discrete-Euler-pure}) and defined by (\ref{eqn:fully-discrete-solution-extended}).
Then there exist positive constants $\const{\infty}$ and $\const{2}$, independent of $\themesh_n$, $\mbf{h}$, $\pdesol$ and $\udg$, for any $n=1,\dots, N$ such that
\begin{eqnarray}
\ltwo{e}{L^\infty(0,T;L^2(\dom))} &\le & \const{\infty} \; \Bigg( \ltwo{e(0)}{} + {\left( \sum_{n=1}^{N} \left( \gamma_{\infty,n} +\eta_{\infty,n}+\beta_{\infty,n} \right) \timestep_n \right)}^{\frac{1}{2}} \nonumber \\
\phantom{\ltwo{e}{L^\infty(0,T;L^2(\dom))}} &\phantom{\le} & \phantom{\const{\infty} \; \Bigg(  }
+ {\left( \sum_{n=1}^{N} \tilde{\eta}_{\infty,n} \right)}^{\frac{1}{2}} 
+ \max_{0 \le n \le N} \{{\ellest \left(\themesh_n, \udg^n, \ellrecsource^n \right)}\} \Bigg) \label{apost_parabolic_boundLinfty}  \\
\ltwo{e}{L^2(0,T;L^2(\dom))} &\le & \const{2} \; \Bigg(   \ltwo{e(0)}{} + {\left( \sum_{n=1}^{N} \left( \gamma_{2,n} +\eta_{2,n}+\beta_{2,n} \right) \; \timestep_n \right)}^{\frac{1}{2}} \nonumber \\
\phantom{\ltwo{e}{L^2(0,T;L^2(\dom))} }&\phantom{\le }& 
\phantom{\const{2} \; \Bigg(   \ltwo{e(0)}{} + \sum_{n=1}^{N} \Bigg( \eta_{1,2,n} } + {\left( \sum_{n=1}^{N} { \; {\ellest \left(\themesh_n, \udg^n, \ellrecsource^n \right)}^2 \; \timestep_n }\right)}^{\frac{1}{2}} \Bigg) . \label{apost_parabolic_boundL2}
\end{eqnarray}

\end{theorem}

The proof of this theorem will be the content of the remaining of this section, split into a number of intermediate results.

We begin by defining the \emph{elliptic reconstruction} $\ellrec^n \in H^2_0(\dom)$, of $\udg^n$ to be
the solution of the elliptic problem
\begin{equation}\label{ell_rec}
B^n(\ellrec^n, v) = \langle \ellrecsource^n, v \rangle\quad \forall v \in H^2_0(\dom)
\end{equation}
where, as above, $\ellrecsource^n := A^n \udg^n-\Pi^n \pdesourcetimeapprox^n +\pdesourcetimeapprox^n$. 
We note that under the assumptions on the domain $\Omega$, we also have $\ellrec^n \in H^4(\dom)$.
We also extend the elliptic reconstruction into a continuous piecewise linear-in-time function
\begin{equation} \label{eqn:fully-discrete-ellrec-extended}
  \ellrec(t):= \sfrac{t - t^{n-1}}{\lambda_n} \ellrec^n + \sfrac{t^{n} - t}{\lambda_n} \ellrec^{n-1}
\end{equation}
for $t \in (t_{n-1},t_{n}]$ and $n=1,\dots, N$.
Finally, we introduce the error decomposition
\begin{equation}\label{errorsplitting}
e:=\udg-\pdesol=\parerr-\ellerr,\ \text{where}\ \ellerr:=\ellrec-\udg,\ \text{and}\ \parerr:=\ellrec-\pdesol,
\end{equation}
where $\parerr$ and $\ellerr$ are understood as the \emph{parabolic} and \emph{elliptic} error, respectively, and we set $\ellerr_n:=\ellerr(t_n)$.

\begin{lemma}[Error identity] \label{lemma:parerridentity}
For all $t \in (t_{n-1},t_{n}]$, $n=1,2,\dots,N$, we have
\begin{equation} \label{eqn:parerridentity}
\langle \parerr_t,v\rangle + B^n(\parerr,v) =\langle \ellerr_t,v\rangle + \langle (I - \Pi^n) \udg_t ,v\rangle + \sfrac{t- t^n}{\timestep_n}\langle \ellrecsource^n - \ellrecsource^{n-1},v\rangle +\langle \pdesourcetimeapprox^n - \pdesource ,v\rangle,
\end{equation}
for any $v\in H^2_0(\dom)$, with $I$ denoting the identity mapping.
\end{lemma}
\begin{proof}
Firstly, from (\ref{ell_rec}) and (\ref{ell_rec2}) we have
\begin{equation} \label{eqn:parerridentity1}
B^n(\ellrec^n,v) - \langle \pdesourcetimeapprox^n ,v\rangle = B^n(\udg^n,\Pi^n v) - \langle \Pi^n \pdesourcetimeapprox^n , \Pi^n v\rangle.
\end{equation}
Also, using the method (\ref{eqn:fully-discrete-Euler-pure}) and the definition of the $L^2$-projection we deduce
\begin{equation}\label{eqn:parerridentity2}
\langle \udg_t ,v\rangle  
= \langle (I - \Pi^n) \udg_t ,v\rangle +  \langle \udg_t ,\Pi^n v\rangle
 = \langle (I - \Pi^n) \udg_t ,v\rangle - ( B^n(\udg_n,\Pi^n v) - \langle \Pi^n\pdesourcetimeapprox^n ,\Pi^n v\rangle) .
\end{equation}
For the elliptic reconstruction error we also have,
\begin{equation}\label{eqn:parerridentity3}
B^n(\ellrec - \ellrec^n,v) 
=  \sfrac{t- t^n}{\timestep_n}\langle \ellrecsource^n - \ellrecsource^{n-1},v\rangle .
\end{equation}
Lastly, for the terms on the left hand side of (\ref{eqn:parerridentity}), we compute
\begin{equation}\label{eqn:parerridentity4}
\langle e_t,v\rangle + B^n(\parerr,v) 
= \langle \udg_t,v\rangle + B^n(\ellrec ,v) - \left( \langle \pdesol_t,v\rangle + B^n(\pdesol,v) \right)
= \langle \udg_t,v\rangle + B^n(\ellrec ,v) -  \langle \pdesource ,v\rangle.
\end{equation}
Using (\ref{eqn:parerridentity1}), (\ref{eqn:parerridentity1}), and (\ref{eqn:parerridentity3}) in (\ref{eqn:parerridentity4}), along with the identity $e=\parerr-\ellerr$ completes the proof.
\end{proof}

The a posteriori bounds (\ref{apost_parabolic_boundLinfty}) and (\ref{apost_parabolic_boundL2}) will be derived by selecting special test functions $v$ in the energy identity  (\ref{eqn:parerridentity}) above, along with estimation of the terms on the right-hand side of (\ref{eqn:parerridentity}).
More specifically, we consider the following two test functions, $\testfuncLinfty:=\parerr$ for the $L^\infty(L^2)$ case, and  
\begin{equation} \label{def:testfuncLtwo}
\testfuncLtwo (t,\cdot) := \int_t^{T} \parerr(s,\cdot) \ud s,\quad t\in[0,T],
\end{equation}
for the $L^2(L^2)$ case; this choice is motivated by Baker \cite{baker:76}, who used a similar construction for the proof of a priori bounds for the second order wave problem. The latter test function has most notably the following properties:
\begin{equation}\label{prop:testfuncLtwo}
\begin{aligned}
&\testfuncLtwo \in H^4(\dom) \cap H^2_0(\dom) \quad \text{ as } \parerr \in H^4(\dom) \cap H^2_0(\dom)\; \text{a.e. in}\ [0,T] ,\\
&\testfuncLtwo (T,\cdot)=0=\Delta \testfuncLtwo (T,\cdot),\quad \nabla \testfuncLtwo(T,\cdot)=0,\quad\text{and}\quad \\
&\testfuncLtwo_t(t,\cdot)=-\parerr(t,\cdot),\quad \text{a.e. in}\ [0,T].
\end{aligned}
\end{equation}

Next, we consider two auxiliary functions which are needed in the consequent proofs. More specifically, on each interval $t \in (t_{n-1},t_{n}]$, for $n=1,\dots, N$, we define
\begin{equation} \label{def:auxiliaryfunc}
\tilde{G}(t) := (I-\Pi^n)\udg + \psi^{n}, \quad\text{ with }\quad \psi^{n}:= -(I-\Pi^n)\udg^{n-1}+\psi^{n-1},\quad \psi^0 :=0,
\end{equation}
and
\begin{equation} \label{def:auxiliaryfunc2}
G(t) := \sfrac{\timestep_{n}}{2} \left( \sfrac{t-t^{n}}{\timestep_{n}} \right)^2 (\ellrecsource^n - \ellrecsource^{n-1}) + \theta^{n}
, \quad\text{ with }\quad\theta^{n}:=  - \sfrac{\timestep_{n}}{2} (\ellrecsource^n - \ellrecsource^{n-1}) + \theta^{n-1},\quad \theta^0=0.
\end{equation}
We note that, for each  $n=1,\dots, N$, we then have $\tilde{G}(t^{n})=\psi^{n}$,  ${{G}}(t^{n})=\theta^{n}$,
\begin{equation} \label{prop:auxiliaryfunc}
\tilde{G}_t(t) := (I-\Pi^n)\udg_t, \quad \text{ and } \quad {G}_t(t) := \sfrac{t- t^n}{\timestep_n} \left(\ellrecsource^n - \ellrecsource^{n-1} \right) .
\end{equation}

The following estimates will be used in the proof of Theorem \ref{theorem:apost_parabolic}.

\begin{lemma} \label{lemma:linftyextratermsestimate}
Let $\tau \in (0,T]$. Then, we have
\begin{eqnarray} 
\int_{0}^{\tau} \langle {\tilde{G}}_t ,\parerr\rangle \; \ud t &\le& \sum_{n=1}^{N} \ltwo{(\Pi^n - I) \udg^{n-1}}{} \max_{0 \le t \le T} \ltwo{\parerr}{} \label{eqn:linftyextratermsestimate1} \\
\int_{0}^{\tau} \langle G_t ,\parerr\rangle \; \ud t &\le&  \sum_{n=1}^{N} \timestep_{n} \ltwo{\ellrecsource^{n} - \ellrecsource^{n-1}}{}
\max_{0 \le t \le T} \ltwo{\parerr}{}  \label{eqn:linftyextratermsestimate2} \\
\int_{0}^{\tau} \langle \pdesourcetimeapprox^n - \pdesource ,\parerr\rangle \; \ud t &\le&  \int_{0}^{\tau}  \ltwo{\pdesourcetimeapprox^n - \pdesource }{} \ud t  \max_{0 \le t \le T} \ltwo{\parerr}{}  \label{eqn:linftyextratermsestimate3} .
\end{eqnarray}
\end{lemma}
\begin{proof}
The proofs of these estimates are immediate via Cauchy-Schwarz-in-space and H\"older-in-time inequalities. 
\end{proof}

In the following three lemmata, we prove bounds for the corresponding terms to the ones in Lemma \ref{lemma:linftyextratermsestimate} when testing with $\testfuncLtwo$ given in (\ref{def:testfuncLtwo}).
\begin{lemma} \label{lemma:ltewoextratermsestimate1}
With the above notation, we have
\begin{equation} 
\label{eqn:ltwoextratermsestimate1} 
\sum_{n=1}^{N} \int_{t^{n-1}}^{t^{n}}\!\! \langle (I - \Pi^n) \udg_t , \testfuncLtwo \rangle \; \ud t \le \sum_{n=1}^{N} \Big( \timestep_n {\ltwo{(I - \Pi^n) \udg^{n-1}}{}}^2   
 + \timestep_n \ltwo{ \sum_{i=1}^{n-1} (\Pi^i - \Pi^{i-1}) \udg^{i-1} }{}^2 \Big)^{\frac{1}{2}} \Big(\int_{t^{n-1}}^{t^{n}}  \ltwo{\parerr}{}^2\ud t \Big)^{\frac{1}{2}}  .  
\end{equation}
\end{lemma}
\begin{proof}
Recalling the definition of $\tilde{G}$, an integration by parts with respect to time gives
\begin{equation} 
\label{eqn:ltwoextratermsestimateproof1} 
\sum_{n=1}^{N} \int_{t^{n-1}}^{t^{n}}  \langle (I - \Pi^n) \udg_t , \testfuncLtwo \rangle \; \ud t  = \sum_{n=1}^{N} \left[ \langle \tilde{G}(t) , \testfuncLtwo(t) \rangle \right]_{t^{n-1}}^{t^{n}}  + \sum_{n=1}^{N} \int_{t^{n-1}}^{t^{n}} \langle \tilde{G} , - \testfuncLtwo_t \rangle \; \ud t  
= \sum_{n=1}^{N} \int_{t^{n-1}}^{t^{n}} \langle \tilde{G} , - \testfuncLtwo_t \rangle \; \ud t.  
\end{equation}
We recall the properties of $\testfuncLtwo$ in (\ref{prop:testfuncLtwo}) and we estimate theright-hand term further:
\begin{equation} 
\label{eqn:ltwoextratermsestimateproof2} 
\sum_{n=1}^{N} \int_{t^{n-1}}^{t^{n}} \langle \tilde{G} , - \testfuncLtwo_t \rangle \; \ud t 
 \le \sum_{n=1}^{N} \Big( \int_{t^{n-1}}^{t^{n}} \ltwo{ \tilde{G}}{}^2 \; \ud t \Big)^{\frac{1}{2}} \Big( \int_{t^{n-1}}^{t^{n}}  \ltwo{\parerr}{}^2 \; \ud t \Big)^{\frac{1}{2}} .  
\end{equation}
The assertion then follows by estimation of the time integral of $\ltwo{ \tilde{G}}{}^2$:
\begin{equation} 
\label{eqn:ltwoextratermsestimateproof3} 
\int_{t^{n-1}}^{t^{n}} \ltwo{ \tilde{G}}{}^2\ud t \le \timestep_n \ltwo{ (I-\Pi^n) \udg^{n-1} }{}^2 + \timestep_n \ltwo{ \psi_{n-1} }{}^2,
\end{equation}
and noting that
$
\psi_{n-1} = \sum_{i=1}^{n-1} (\Pi^i - \Pi^{i-1}) \udg^{i-1}
$.
\end{proof}


\begin{lemma} \label{lemma:ltewoextratermsestimate2}
With the above notation, we have
\begin{equation} 
\label{eqn:ltwoextratermsestimate2} 
\sum_{n=1}^{N}\!\! \int_{t^{n-1}}^{t^{n}}    \sfrac{t- t^n}{\timestep_n}\langle\ellrecsource^n - \ellrecsource^{n-1} , \testfuncLtwo \rangle \; \ud t \le \sum_{n=1}^{N} \Big( \timestep_n^3 \ltwo{\ellrecsource^n - \ellrecsource^{n-1} }{}^2    {+ \timestep_n \ltwo{ \sum_{i=1}^{n-1} \sfrac{\timestep_{i}}{2} (\ellrecsource^i - \ellrecsource^{i-1})  }{}^2 \Big)}^{\frac{1}{2}} \Big( \int_{t^{n-1}}^{t^{n}}  \ltwo{\parerr}{}^2\ud t \Big)^{\frac{1}{2}}  .  
\end{equation}
\end{lemma}
\begin{proof}
Recalling the definition of $G$, an integration by parts with respect to time gives
\begin{equation} 
\label{eqn:ltwoextratermsestimateproof21} 
\sum_{n=1}^{N} \int_{t^{n-1}}^{t^{n}}   \sfrac{t- t^n}{\timestep_n}\langle \ellrecsource^n - \ellrecsource^{n-1} , \testfuncLtwo \rangle  \ud t  
= \sum_{n=1}^{N} \int_{t^{n-1}}^{t^{n}} \langle G , - \testfuncLtwo_t \rangle \; \ud t.  
\end{equation}
We recall the properties of $\testfuncLtwo$ in (\ref{prop:testfuncLtwo}) and estimate the right-hand side further:
\begin{equation} 
\label{eqn:ltwoextratermsestimateproof22} 
\sum_{n=1}^{N} \int_{t^{n-1}}^{t^{n}} \langle G , - \testfuncLtwo_t \rangle \; \ud t 
\le \sum_{n=1}^{N} {\Big( \int_{t^{n-1}}^{t^{n}} \ltwo{ G}{}^2 \; \ud t \Big)}^{\frac{1}{2}} {\Big( \int_{t^{n-1}}^{t^{n}}  \ltwo{\parerr}{}^2 \; \ud t \Big)}^{\frac{1}{2}} .  
\end{equation}
The assertion then follows by estimation of the integral of $\ltwo{ G}{}^2$:
\begin{equation} 
\label{eqn:ltwoextratermsestimateproof23} 
\int_{t^{n-1}}^{t^{n}} \ltwo{G}{}^2\ud t \le \timestep_n^3 \ltwo{\ellrecsource^n - \ellrecsource^{n-1} }{}^2 + \timestep_n \ltwo{ \theta_{n-1} }{}^2
\end{equation}
and noting that $
\theta_{n-1} = \sum_{i=1}^{n-1} - \sfrac{\timestep_{i}}{2} (\ellrecsource^i - \ellrecsource^{i-1})$.
\end{proof}

\begin{lemma} \label{lemma:ltewoextratermsestimate3}
With the above notation, we have
\begin{equation} 
\label{eqn:ltwoextratermsestimate3} 
\begin{aligned}
\sum_{n=1}^{N} \int_{t^{n-1}}^{t^{n}}  \langle \pdesourcetimeapprox^n - \pdesource  , \testfuncLtwo \rangle \; \ud t &\le C_{\text{app}} \sum_{n=1}^{N} \Big(\int_{t^{n-1}}^{t^{n}} \timestep_n^2 \ltwo{ \pdesourcetimeapprox^n - \pdesource}{}^2\ud t\Big)^{1/2} \Big(\int_{t^{n-1}}^{t^n} \ltwo{\parerr}{}^2\ud t\Big)^{1/2}.  
\end{aligned}
\end{equation}
\end{lemma}
\begin{proof}

As $\pdesourcetimeapprox^n$ is the $L^2$-projection of $\pdesource$ in time, we have
\begin{equation}
\sum_{n=1}^{N} \int_{t^{n-1}}^{t^{n}}  \langle \pdesourcetimeapprox^n - \pdesource,\testfuncLtwo \rangle\ud t
=
\sum_{n=1}^{N} \int_{t^{n-1}}^{t^{n}} \langle \pdesourcetimeapprox^n - \pdesource,\testfuncLtwo - \zeta^n\rangle \ud t,
\end{equation}
for the lowest order time approximation $\zeta^n(\cdot):=\timestep_n^{-1}\int_{t^{n-1}}^{t^n}\testfuncLtwo(t,\cdot) \ud t$ of $\testfuncLtwo$. 
With the approximation property of $\zeta^n$ in time, we deduce
\begin{equation}
\int_{t^{n-1}}^{t^n}\ltwo{\testfuncLtwo - \zeta^n}{}^2\ud t \le C_{\text{app}}^2 \timestep_n^2 \int_{t^{n-1}}^{t^n}\ltwo{\testfuncLtwo_t}{}^2\ud t,
\end{equation}
and recalling that $\testfuncLtwo_t=-\parerr$, the Cauchy-Schwarz inequality implies
\begin{equation}
\sum_{n=1}^{N} \int_{t^{n-1}}^{t^{n}}  \langle \pdesourcetimeapprox - \pdesource,\testfuncLtwo \rangle\ud t
\le C_{\text{app}}
\sum_{n=1}^{N} \Big(\int_{t^{n-1}}^{t^{n}} \timestep_n^2 \ltwo{ \pdesourcetimeapprox^n - \pdesource}{}^2\ud t\Big)^{1/2} \Big(\int_{t^{n-1}}^{t^n} \ltwo{\parerr}{}^2\ud t\Big)^{1/2}.
\end{equation}

\end{proof}

To complete a posteriori error bounds in Theorem \ref{theorem:apost_parabolic}, we also need the following two lemmata in which the elliptic error terms $\ellerr$ and $\ellerr_t$ are estimated by fully computable residuals.
\begin{lemma} \label{lemma:ellipticerrorsestimate1}
Let $\ellerr$ be as in (\ref{errorsplitting}). Then, we have 
\begin{equation} 
\label{eqn:ellipticerrorsestimate1} 
\int_{0}^{T} \ltwo{\ellerr}{}^2 \; \ud t \le \sfrac{2\timestep_N}{3} \ellest(\themesh_N, \udg^N,\ellrecsource^N) + \sum_{n=1}^{N-1} \sfrac{4\timestep_n}{3}  \ellest \left(\themesh_n, \udg^n, \ellrecsource^n \right)^2 
\end{equation}
\end{lemma}
\begin{proof}
Noting that $((t- t^{n-1})/\timestep_n)^2\le \frac{1}{3}$ and $((t^{n}-t)\timestep_n)^2\le \frac{1}{3}$, we have
\begin{equation}
 \int_{0}^{T} \ltwo{\ellerr}{}^2  \ud t
 \le  \sum_{n=1}^{N} \sfrac{2\timestep_n}{3} \Big( \ltwo{\ellerr_{n}}{}^2 + \ltwo{\ellerr_{n-1}}{}^2 \Big).
\end{equation}
The assertion then follows by Theorem \ref{theorem:ltwo_apost_biharmonic}.
\end{proof}
\begin{lemma} \label{lemma:ellipticerrorsestimate2}
Let $\ellerr$ be as in (\ref{errorsplitting}) and $\tau \in [0,T]$; then, we have
\begin{equation} 
\label{eqn:ellipticerrorsestimate2} 
\int_{0}^{\tau} \langle \ellerr_t , \parerr \rangle  \ud t  \le \sum_{n=1}^{N}
\ellest (\hat{\themesh}_n, \udg^n-\udg^{n-1},\ellrecsource^n - \ellrecsource^{n-1} )\max_{0 \le t \le T} \ltwo{\parerr}{} ,
\end{equation}
where 
 $\hat{\themesh}_n := \themesh_{n} \cap \themesh_{n-1}$ denotes the finest common coarsening of $\themesh_{n}$ and $\themesh_{n-1}$, $n=1,\dots, N$.
\end{lemma}
\begin{proof}
We have
$ \ellerr_t(t) = (\ellerr_n-\ellerr_{n-1})/\timestep_n$, for $ t \in (t_{n-1},t_{n}]$ and  $n=1,\dots, N$.
Denoting $\tau := t_{r+1/2}$ and $r:=\max \{k : t_k \le \tau, k=1,\ldots, N \}$, we then have
\begin{equation} \label{eqn:ellipticerrorsestimate_proof1}
\begin{aligned}
\int_{0}^{\tau} \langle \ellerr_t , \parerr \rangle  \ud t  &= \sum_{n=1}^{r+1/2} \int_{t^{n-1}}^{t^{n}} \sfrac{1}{\timestep_n} \langle \ellerr_n-\ellerr_{n-1} , \parerr \rangle  \ud t 
\le  \max_{0 \le t \le T} \ltwo{\parerr}{}\sum_{n=1}^{r+1/2} \ltwo{\ellerr_n-\ellerr_{n-1}}{}.
\end{aligned}
\end{equation}
We now observe that the finite element function $\tilde{z}$ in the proof of Theorem \ref{theorem:ltwo_apost_biharmonic} can be selected from a subspace of $\femspace$: in particular, we can select the finite element subspace corresponding to the finest common coarsening mesh $\hat{\themesh}_n $, for $n=1,\dots,N$. Then, following completely analogous argument as in the proof of  of Theorem \ref{theorem:ltwo_apost_biharmonic}, we can arrive to the bound
\[
\ltwo{\ellerr_n-\ellerr_{n-1}}{} \le  \ellest \Big(\hat{\themesh}_n, \udg^n-\udg^{n-1},\ellrecsource^n - \ellrecsource^{n-1} \Big),
\]
which already yields the result.
\end{proof}

\begin{remark} \label{remark:ellertimederivativebound}
Note that the following simpler alternative bound for the term in Lemma \ref{lemma:ellipticerrorsestimate2} is also possible,
\begin{equation} 
\label{eqn:ellipticerrorsestimate2alt} 
\begin{aligned}
\int_{0}^{\tau} \langle \ellerr_t , \parerr \rangle  \ud t & \le  \; \sum_{n=0}^{N}  \ellest \left(\themesh_n, \udg^n, \ellrecsource^n \right)  \; \max_{0 \le t \le T}  \ltwo{\parerr}{} .
\end{aligned}
\end{equation}
This bound shifts the emphases from the finest common coarsening mesh, $\hat{\themesh}_n$, in Lemma \ref{lemma:ellipticerrorsestimate2} to the elliptic estimators acting on meshes at each time step only which can be of practical importance when implementing adaptive algorithms based on the estimators.
\end{remark}

\begin{proof} {\bf of Theorem \ref{theorem:apost_parabolic}} To conclude the proof, 
we estimate the left-hand side of (\ref{eqn:parerridentity}) in each case of the test functions: $\testfuncLtwo$ to derive $L^2(L^2)$-norm a posteriori bound and $\parerr$ for the $L^\infty(L^2)$-norm bound.
First we deal with the $L^2(L^2)$ case; we start by integrating (\ref{eqn:parerridentity}) by parts in time,
\begin{equation}\label{eqn:apost_parabolic_ltwo_parerridentity}
\begin{aligned}
\int_{0}^{T}  \langle e_t , \testfuncLtwo \rangle  + B(\parerr,\testfuncLtwo) \ud t& = \int_{0}^{T} \langle e , -\testfuncLtwo_t \rangle  \ud t + \left[ \langle e ,\testfuncLtwo \rangle \right]_{0}^{T} - \int_{0}^{T}  B(\testfuncLtwo_t,\testfuncLtwo) \ud t \\
& = \int_{0}^{T} \langle \parerr , \parerr \rangle  \ud t - \int_{0}^{T} \langle \ellerr , \parerr \rangle  \ud t - \langle e(0) ,\testfuncLtwo(0) \rangle  - \int_{0}^{T} \sfrac{1}{2} \sfrac{d}{dt}  B(\testfuncLtwo,\testfuncLtwo) \ud t \\
& = \int_{0}^{T} \ltwo{\parerr}{}^2 \ud t - \int_{0}^{T} \langle \ellerr , \parerr \rangle  \ud t - \langle e(0) ,\testfuncLtwo(0) \rangle  + \sfrac{1}{2} B(\testfuncLtwo(0),\testfuncLtwo(0))  .
\end{aligned}
\end{equation}
We also have
\begin{equation}\label{eqn:apost_parabolic_ltwo_energyterm}
\langle e(0) ,\testfuncLtwo(0) \rangle \le \ltwo{e(0)}{} \ltwo{\testfuncLtwo(0)}{} \le \ltwo{e(0)}{} C_{\text{pf}} B(\testfuncLtwo(0),\testfuncLtwo(0)).
\end{equation} 
Using (\ref{eqn:apost_parabolic_ltwo_parerridentity}) and (\ref{eqn:apost_parabolic_ltwo_energyterm})
in (\ref{eqn:parerridentity}) after integration over each interval $(t^{n-1},t^{n}]$ and summation with respect to $n$, we get,
\begin{equation}
\ltwo{\parerr}{L^2(0,T,L^2(\dom))}^2 \le  \ltwo{e(0)}{}^2 + \sum_{n=1}^{N} \int_{t^{n-1}}^{t^{n}} \Big( \langle \ellerr , \parerr \rangle  + \langle (I - \Pi^n) \udg_t ,\testfuncLtwo\rangle  + \sfrac{t- t^n}{\timestep_n }\langle \ellrecsource^n - \ellrecsource^{n-1} ,\testfuncLtwo\rangle +\langle \pdesourcetimeapprox^n - \pdesource ,\testfuncLtwo\rangle \Big)\ud t .
\end{equation}
The bound (\ref{apost_parabolic_boundL2}) now follows upon using the triangle inequality 
\[
\ltwo{e}{L^2(0,T,L^2(\dom))} \le \ltwo{\parerr}{L^2(0,T,L^2(\dom))} + \ltwo{\ellerr}{L^2(0,T,L^2(\dom))},
\]
Young's inequality and Lemmata \ref{lemma:ltewoextratermsestimate1}, \ref{lemma:ltewoextratermsestimate2}, \ref{lemma:ltewoextratermsestimate3} and \ref{lemma:ellipticerrorsestimate1}.

For the $L^\infty(L^2)$-norm case, upon testing with $v=\parerr$, we deduce for the left-hand side of (\ref{eqn:parerridentity}) by integrating by parts to some $\tau \in [0,T]$,
\begin{equation}\label{eqn:apost_parabolic_linfty_parerridentity}
\int_{0}^{\tau}  \langle e_t , \parerr \rangle  + B(\parerr,\parerr) \ud t
 = \ltwo{ \parerr(\tau)}{}^2 - \ltwo{ \parerr(0)}{}^2 - \int_{0}^{\tau} \langle \ellerr_t , \parerr \rangle  \ud t+ \int_{0}^{\tau}  B(\parerr,\parerr) \ud t .
\end{equation}
Choosing $\tau$ such that $\ltwo{ \parerr(\tau)}{}=\max_{0 \le t \le T}  \ltwo{\parerr}{} $, using the triangle inequality,
$\ltwo{\parerr(0)}{} \le \ltwo{e(0)}{} + \ltwo{\ellerr(0)}{}$,
 and (\ref{eqn:apost_parabolic_linfty_parerridentity}) in (\ref{eqn:parerridentity}), we get,
\begin{equation}
\ltwo{\parerr}{L^\infty(0,T,L^2(\dom))}^2 +\int_{0}^{\tau} B(\parerr,\parerr) \ud t  \le \ltwo{e(0)}{}^2+\ltwo{\ellerr(0)}{}^2 + \int_{0}^{\tau} \Big( \langle \ellerr_t , \parerr \rangle  + \langle \tilde{G}_t ,\parerr\rangle 
 + \langle {G}_t,\parerr\rangle +\langle \pdesourcetimeapprox^n - \pdesource ,\parerr\rangle \Big)\ud t 
\end{equation}
where $G$ and $\tilde{G}$ are given by (\ref{def:auxiliaryfunc2}) and (\ref{def:auxiliaryfunc}).
The bound (\ref{apost_parabolic_boundLinfty}) follows again using triangle inequality 
\begin{equation}
\ltwo{e}{L^\infty(0,T,L^2(\dom))} \le \ltwo{\parerr}{L^\infty(0,T,L^2(\dom))} + \ltwo{\ellerr}{L^\infty(0,T,L^2(\dom))},
\end{equation}
Lemmata \ref{lemma:linftyextratermsestimate} and \ref{lemma:ellipticerrorsestimate2} as well as 
$\max_{0 \le t \le T}  \ltwo{\ellerr}{} \le \max_{0 \le t \le T}  \ellest \left(\themesh_n, \udg^n, \ellrecsource^n \right)$.
\end{proof}

We note that a posteriori bounds in the $L^2(H^2)$-norm of the error have been already considered in \cite{thesis_juha}, along with their application within an adaptive algorithm. The $L^2(H^2)$-norm theoretical and numerical results appear to be of the expected order of convergence; they are omitted here for brevity.

\section{Numerical Experiments}\label{numerics}

For $t\in [0,1]$ and $\dom := (0,1)^2$, we consider two benchmark problems for which $u_0$ and $f$ are chosen so that the exact solution $u$ of problem (\ref{varmodelpde}) coincides with one of the following solutions:
\begin{gather}
   \label{eqn:numerics:example1} 
   u_1(x,y,t) 
   = \sin(\pi t) \; 10^2 \; \sin^2(\pi x) \; \sin^2(\pi y) e^{-10(x^2+y^2)},
   \\
  \label{eqn:numerics:example2}
   u_2(x,y,t) 
   = \sin(20 \pi t) \sin^2(\pi x)\sin^2(\pi y) e^{-10(x^2+y^2)}.
\end{gather}

Solutions $u_1$ and $u_2$ are both smooth but $u_2$ oscillates much faster where as $u_1$ exhibits greater space dependency of the error. They are defined so as to emphasize different aspects of the estimators at hand. Similar examples have been studied elsewhere, for example in
\cite{thesis_juha} in the context of $L^2(H^2)$-norm a posteriori estimators; see also \cite{lakkis-makridakis:06,LakkisPryer:10,georgoulis_lakkis_virtanen} for similar examples in the context of second order problems.

For the numerical experiments, the library FEniCS ({\tt http://fenicsproject.org/}) was used. For each of the examples, we compute the solution of (\ref{eqn:fully-discrete-Euler-pure}) using quadratic simplicial finite element spaces and with interior penalty parameters $\sigma_0=\xi_0=20$ in (\ref{penaltyparametersdef}), which is sufficient to guarantee stability of the numerical scheme. The interior penalty parameters have a known effect on the effectivity indices, cf., \cite{thesis_juha,karakashian-pascal:07}.

We study the asymptotic behavior of the indicators by setting all constants appearing in Theorem
\ref{theorem:apost_parabolic} equal to one. We monitor the evolution of the values and the experimental order of
convergence of the estimators and the error as well as of the effectivity
index over time on a sequence of uniformly refined meshes with 
$
	h_{\kappa,i}:=2^{-i/2-1}$, $i=1,\dots,5$, $\kappa \in \themesh
$
with fixed time steps $\timestep \approx \max_\kappa  h_\kappa ^3 $ and $\timestep \approx \max_\kappa h_\kappa^2 $. To this end, we define \emph{experimental order of convergence} ($EOC$) of
a given sequence of positive quantities $a(i)$ defined on a sequence
of meshes of size $h(i)$ by
\begin{equation}
  EOC( a,i ) = \frac{\log(a(i+1)/a(i))}{\log(h(i+1)/h(i))},
\end{equation}
the accumulated coarsening or mesh change estimators by
\begin{equation}\label{acc_apost_parabolic_coarse_estimators}
 \mathbb{E}_{\text{coarsen},\infty,m}  :=\Big( \sum_{n=1}^{m} \gamma_{\infty,n} \timestep_n \Big)^{\frac{1}{2}}\quad \text{ and }\quad
 \mathbb{E}_{\text{coarsen},2,m}  := \Big(\sum_{n=1}^{m} \gamma_{2,n} \timestep_n \Big)^{\frac{1}{2}},
\end{equation}
accumulated time error evolution estimators by
\begin{equation}\label{acc_apost_parabolic_time_estimators}
 \mathbb{E}_{\text{time},\infty,m}  := \Big(\sum_{n=1}^{m} ( \eta_{\infty,n} + \beta_{\infty,n}) \timestep_n + \sum_{n=1}^{m}  \tilde{\eta}_{\infty,n}\Big)^{\frac{1}{2}}  \quad \text{ and } \quad
 \mathbb{E}_{\text{time},2,m}  := \Big(\sum_{n=1}^{m} ( \eta_{2,n} + \beta_{2,n}) \timestep_n\Big)^{\frac{1}{2}} ,
\end{equation}
accumulated space error estimators by
\begin{equation}\label{acc_apost_elliptic_space_estimators}
 \mathbb{E}_{\text{space},\infty,m}  := \max_{0 \le n \le N} \{{\ellest \left(\themesh_n, \udg^n, \ellrecsource^n \right)}\} \quad \text{ and } \quad
 \mathbb{E}_{\text{space},2,m}  := \Big(\sum_{n=1}^{m} {\ellest \left(\themesh_n, \udg^n, \ellrecsource^n \right)}^2 \timestep_n\Big)^{\frac{1}{2}} ,
\end{equation}
and the \emph{inverse effectivity index} 
\begin{equation}
  IEI_m = \frac{\ltwo{e}{L^\infty(0,t_m;L^2(\dom))}}{\mathbb{E}_{\text{time},\infty,m} + \mathbb{E}_{\text{space},\infty,m}} \quad \text{or} \quad  IEI_m = \frac{\ltwo{e}{L^2(0,t_m;L^2(\dom))}}{\mathbb{E}_{\text{time},2,m} + \mathbb{E}_{\text{space},2,m}},
\end{equation}
for the case $L^\infty(L^2)$ and $L^2(L^2)$, respectively. The IEI conveys
the same information as the (standard) effectivity index and has the advantage of relating directly
to the constants appearing in Theorem \ref{theorem:apost_parabolic}.

The results of numerical experiments on uniform meshes, depicted in Figures 1 - 4, indicate that the error estimators 
are reliable and also efficient which can be seen from the effectivity index behaviour and the EOC of the error and the
time and space estimators for both $L^2(L^2)$- and $L^\infty(L^2)$-norm a posteriori bounds.

To further evaluate practical aspects of the derived a posteriori estimators, they are incorporated within in two adaptive algorithms; these are outlined in pseudocode as follows
\hspace{-1.5cm}\\
\begin{tabular}{p{8cm} p{8cm}}
\vspace{0pt} 
{\footnotesize
	\begin{tabular}{llllll} 
	\multicolumn{6}{l}{{\bf ImplicitTimeStepControl}}\\
	\multicolumn{6}{l}{{\bf Input:} $U_{0}, f, \text{TOL}_{\text{time,min}}, \text{TOL}_{\text{time}}, \text{TOL}_{\text{space}}, \ldots $ }\\
	& \multicolumn{5}{l}{ $ \text{TOL}_{\text{coarse}}, \timestep_0, t_0, T, \ldots $} \\
	& \multicolumn{5}{l}{ $ \mathcal{T}_{0}, \xi_{refine}, \text{{\bf SpaceAdaptivity}}, \ldots $} \\
	& \multicolumn{5}{l}{ $ \text{{\bf InitialSpaceAdaptivity}}$} \\
	\multicolumn{6}{l}{\{ Initial condition interpolation and mesh refinement \}} \\
	\multicolumn{6}{l}{$(U_0,\mathcal{T}_{0})$:={\bf InitialSpaceAdaptivity}($U_{0}, f, \mathcal{T}_{0}, \xi_{refine}$).} \\
	\multicolumn{6}{l}{\{Initialize.\}} \\
	\multicolumn{6}{l}{{\bf Set:} $n=1$, $\timestep_n=\timestep_{n-1}$.} \\
	\multicolumn{6}{l}{{\bf While} $(t_n \le T)$ }\\
	& \multicolumn{5}{l}{{\bf Set:} $\mathbb{E}_{\text{time}} := \text{TOL}_{\text{time}} + 1$}\\
	& \multicolumn{5}{l}{{\bf While} $(\mathbb{E}_{\text{time}} > \text{TOL}_{\text{time}})$}\\
	& & \multicolumn{4}{l}{$t_{n}:=t_{n-1}+\timestep_n$}\\
	& & \multicolumn{4}{l}{{\bf Set:} $\mathcal{T}_{t} := \mathcal{T}_{n}$}\\
	& & \multicolumn{4}{l}{ $(U_n,\mathcal{T}_{n})$ := {\bf SpaceAdaptivity}($U_{n-1}, \ldots$ }\\
	& & & \multicolumn{3}{l}{ $\phantom{(U_n,\mathcal{T}_{n}):=}$  $ f, \text{TOL}_{\text{space}}, \text{TOL}_{\text{coarse}}, \ldots $}\\
	& & & \multicolumn{3}{l}{ $\phantom{(U_n,\mathcal{T}_{n}):=}$  $ \timestep_n, t_n, T, \mathcal{T}_{n-1}, \xi_{refine}$) }\\
	& & \multicolumn{4}{l}{ compute $\mathbb{E}_{\text{time}}$ .}\\
	& & \multicolumn{4}{l}{{\bf if } $(\mathbb{E}_{\text{time}} > \text{TOL}_{\text{time}})$ {\bf then }}\\
	& & & \multicolumn{3}{l}{ \{Shorten timestep.\} }\\
	& & & \multicolumn{3}{l}{ $\timestep_{n} := \timestep_n / 2 $ }\\
	& & & \multicolumn{3}{l}{{\bf Set:} $\mathcal{T}_{n} := \mathcal{T}_{t}$}\\
	& & \multicolumn{4}{l}{{\bf endif}}\\
	& \multicolumn{5}{l}{{\bf End While}}\\
	& \multicolumn{5}{l}{ $\timestep_{n+1} := \timestep_n * 2$ }\\
	& \multicolumn{5}{l}{$n:=n+1$}\\
	\multicolumn{6}{l}{{\bf End While}}\\
	\multicolumn{6}{l}{{\bf Output:} $U_n$} \\
	\end{tabular}
  \label{thetimealgorithm2}
}
& \vspace{0pt}

{\footnotesize
	\begin{tabular}{llllll} 
	\multicolumn{6}{l}{{\bf ExplicitTimeStepControl}}\\
	\multicolumn{6}{l}{{\bf Input:} $U_{0}, f, \text{TOL}_{\text{time,min}}, \text{TOL}_{\text{time}}, \text{TOL}_{\text{space}}, \ldots $ }\\
	& \multicolumn{5}{l}{ $ \text{TOL}_{\text{coarse}}, \timestep_0, t_0, T, \ldots $} \\
	& \multicolumn{5}{l}{ $ \mathcal{T}_{0}, \xi_{refine}, \text{{\bf SpaceAdaptivity}}, \ldots $} \\
	& \multicolumn{5}{l}{ $ \text{{\bf InitialSpaceAdaptivity}}$} \\
	\multicolumn{6}{l}{\{ Initial condition interpolation and mesh refinement \}} \\
	\multicolumn{6}{l}{$(U_0,\mathcal{T}_{0})$:={\bf InitialSpaceAdaptivity}($U_{0}, f, \mathcal{T}_{0}, \xi_{refine}$).} \\
	\multicolumn{6}{l}{\{Initialize.\}} \\
	\multicolumn{6}{l}{{\bf Set:} $n=1$, $\timestep_n=\timestep_{n-1}$ and $t_n=t_{n-1} + \timestep_n$.} \\
	\multicolumn{6}{l}{{\bf While} $(t_n \le T)$ }\\
	& \multicolumn{5}{l}{ $(U_n,\mathcal{T}_{n})$ := {\bf SpaceAdaptivity}($U_{n-1}, \ldots$ }\\
	& & \multicolumn{4}{l}{ $\phantom{(U_n,\mathcal{T}_{n}):=}$  $ f, \text{TOL}_{\text{space}}, \text{TOL}_{\text{coarse}}, \ldots $}\\
	& & \multicolumn{4}{l}{ $\phantom{(U_n,\mathcal{T}_{n}):=}$  $ \timestep_n, t_n, T, \mathcal{T}_{n-1}, \xi_{refine}$) }\\
	 & \multicolumn{5}{l}{ compute $\mathbb{E}_{\text{time}}$ .}\\
	 & \multicolumn{5}{l}{{\bf if } $(\mathbb{E}_{\text{time}} > \text{TOL}_{\text{time}})$ {\bf then }}\\
	 & & \multicolumn{4}{l}{ $\timestep_{n+1} := \timestep_n / \sqrt{2} $ }\\
	& \multicolumn{5}{l}{{\bf elseif } $(\mathbb{E}_{\text{time}} < \text{TOL}_{\text{time,min}})$ {\bf then }}\\
	 & & \multicolumn{4}{l}{ $\timestep_{n+1} := \timestep_n * \sqrt{2}$ }\\
	& \multicolumn{5}{l}{{\bf endif}}\\
	& \multicolumn{5}{l}{$t_{n+1}:=t_n+\timestep_n$}\\
	& \multicolumn{5}{l}{$n:=n+1$}\\
	 \multicolumn{6}{l}{ {\bf End While} }\\
	\multicolumn{6}{l}{{\bf Output:} $U_n$} \\
	\end{tabular}
  \label{thetimealgorithm1}
}
\\
\end{tabular}
\noindent where {\bf SpaceAdaptivity} (and {\bf InitialSpaceAdaptivity}) are performed using a standard {\em D\"orfler marking} strategy expressed in pseudocode as follows

\vspace{4pt}
{\footnotesize
\begin{tabular}{llllll} 
\multicolumn{6}{l}{{\bf SpaceAdaptivity}}\\
\multicolumn{6}{l}{{\bf Input:} $U_{n-1}, f, \text{TOL}_{\text{space}}, \text{TOL}_{\text{coarse}}, \tau_n, t_n, T, \mathcal{T}_{n-1}, \xi_{refine}$} \\
\multicolumn{6}{l}{{\bf Set:} $\mathcal{T}_{n} := \mathcal{T}_{n-1}$.} \\
\multicolumn{6}{l}{$\mathcal{T}_{n} := $ \bf{SpaceCoarsening}($U_{n-1},\text{TOL}_{\text{coarse}}, \tau_n, \mathcal{T}_{n}$)} \\
\multicolumn{6}{l}{\{Refinement\}} \\
\multicolumn{6}{l}{compute local elliptic estimators, $(\text{LocalEst}_{n,\kappa})_{\kappa \in \mathcal{T}_{n}}$.}\\
\multicolumn{6}{l}{sum up local estimators and set $\text{Sum}_{\text{total}} := \sum_{\kappa \in \mathcal{T}_{n}} \text{LocalEst}_{n,\kappa}$, and compute $\mathbb{E}_{\text{space}}$.}\\
\multicolumn{6}{l}{{\bf While} $(\mathbb{E}_{\text{space}} > \text{TOL}_{\text{space}})$ }\\
 & \multicolumn{5}{l}{sort $(\text{LocalEst}_{n,\kappa})_{\kappa \in \mathcal{T}_{n}}$ in descending order, set $Q := \emptyset$.}\\
 & \multicolumn{5}{l}{{\bf Set:} $\text{Sum} = 0$.} \\
 & \multicolumn{5}{l}{{\bf While} $(( \text{Sum} < \xi_{refine} * \text{Sum}_{\text{total}} )$ and $(\kappa \in \mathcal{T}_{n}))$  }\\
 & \multicolumn{5}{l}{\{D\"orfler marking \} }\\
 & & \multicolumn{4}{l}{ $\text{Sum} := \text{Sum} + \text{LocalEst}_{n,\kappa}$}\\
 & & \multicolumn{4}{l}{{\bf if} $(Sum < \xi_{refine} * \text{Sum}_{\text{total}})$ }\\
 & & & \multicolumn{3}{l}{Mark $\kappa$ for refinement; $Q := \{ \kappa \} \cup Q$. }\\
 & \multicolumn{5}{l}{{\bf End While} }\\
 & \multicolumn{5}{l}{Refine all elements in $Q$ to obtain new mesh $\mathcal{T}_{n}$. }\\
 & \multicolumn{5}{l}{Solve $I^n U_{n-1}$. }\\
 & \multicolumn{5}{l}{Solve (\ref{eqn:fully-discrete-Euler-pure}) for $U_n$ with $\Pi^n U_{n-1}, \Pi^n f^n, \tau_n$ and $t_n$ on $\mathcal{T}_n$. }\\
 & \multicolumn{5}{l}{compute local elliptic estimators, $(\text{LocalEst}_{n,\kappa})_{\kappa \in \mathcal{T}_{n}}$.}\\
 & \multicolumn{5}{l}{sum up local estimators and set $\text{Sum}_{\text{total}} := \sum_{\kappa \in \mathcal{T}_{n}} \text{LocalEst}_{n,\kappa}$, and compute $\mathbb{E}_{\text{space}}$.}\\
 \multicolumn{6}{l}{ {\bf End While} }\\
\multicolumn{6}{l}{{\bf Output:} $U_n, \mathcal{T}_n$} 
\end{tabular}
}

\noindent The refinement ratio $0 < \xi_{refine} \le 1$ and the tolerances $\text{TOL}_{\text{space}}>0,\text{TOL}_{\text{space}}>0$ and $ \text{TOL}_{\text{coarse}}>0$ are predefined quantities. The value of $\xi_{refine}:=0.75$ was used throughout the experiments in adaptive algorithms. Note that the coarsening tolerance, $\text{TOL}_{\text{coarse}}$, (as well as the tolerance for the alternative space estimator in $L^\infty$ case of Remark \ref{remark:ellertimederivativebound}) had to be determined experimentally for given space and time tolerances and depending on an example. 

The results of experiments with adaptive algorithms as well as a comparison between the two algorithms, are detailed in Figures 6-8 where we monitor time step size, accumulated degrees of freedom and error evolution in comparison to the uniform approach leading to the desired tolerance.  The results of these test cases imply substantial reduction in degrees of freedom by both adaptive algorithms in order to reach the same error tolerance as compared with the uniform approach. This implies a potential efficiency gain in solving PDE problems addressed in this work.

The estimators presented here are found to be suitable for both adaptive time stepping algorithms due to their good separated scaling properties in time and in space. The numerical results appear to be less sensitive to mesh change, compared to the same adaptive algorithms based on the $L^2(H^2)$-norm a posteriori error estimators presented in \cite{thesis_juha}. 
For instance, terms involving $({\ellrecsource}^{n} - {\ellrecsource}^{n-1})$ which is sensitive to mesh change (coarsening as well as refinement) scale down sufficiently fast with the a posteriori estimators presented in this work, resulting to robust error reduction in an adaptive algorithm. 

Finally, we note that the considerably more computationally efficient {\bf ExplicitTimeStepControl} algorithm (due to absence of time step searching step) was found to reach desired error tolerances (even though this is not guaranteed in general) in these numerical experiments.

\begin{figure}[ht]
		\includegraphics[clip,width=\textwidth]{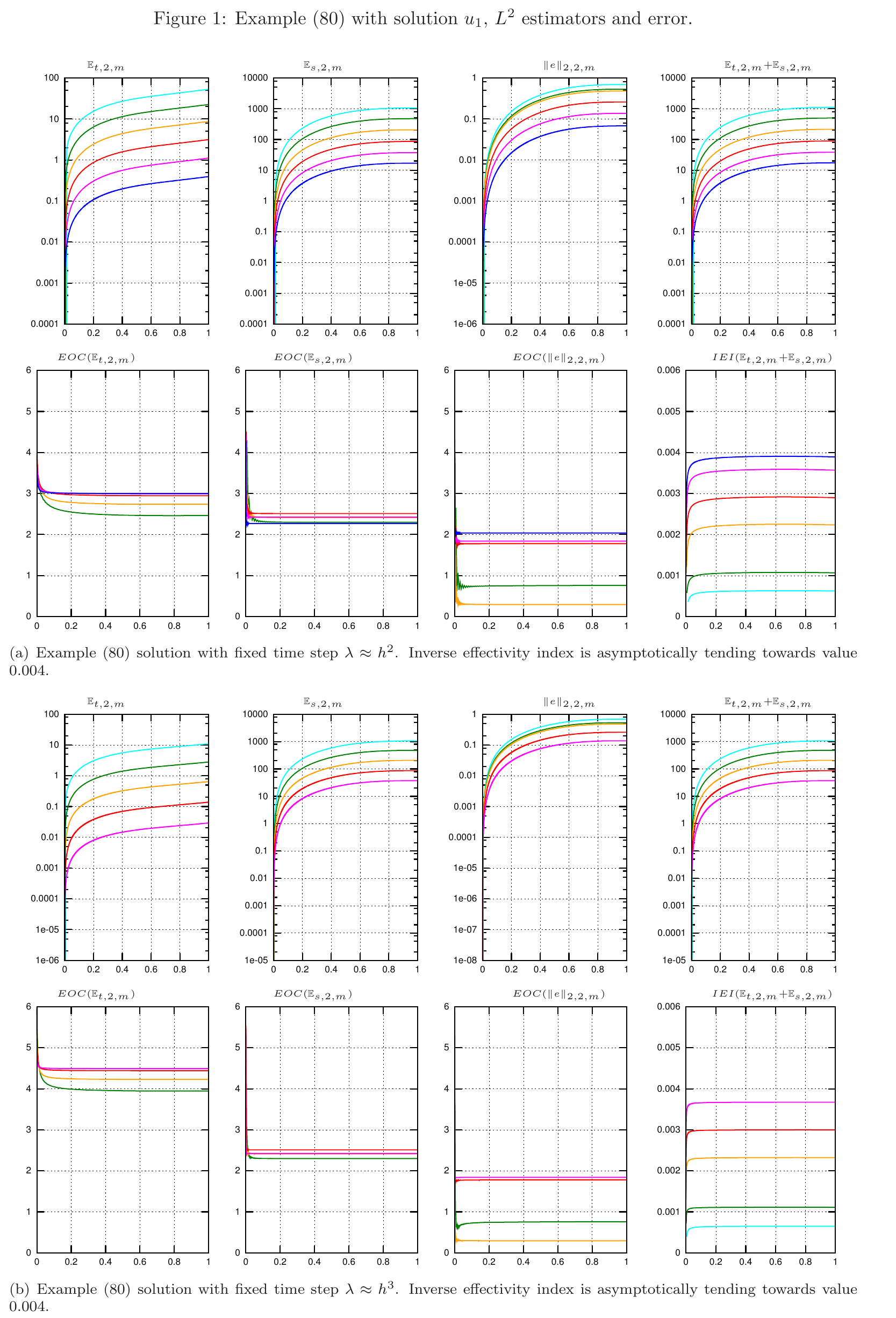}
\end{figure}

\begin{figure}[ht]
		\includegraphics[clip,width=\textwidth]{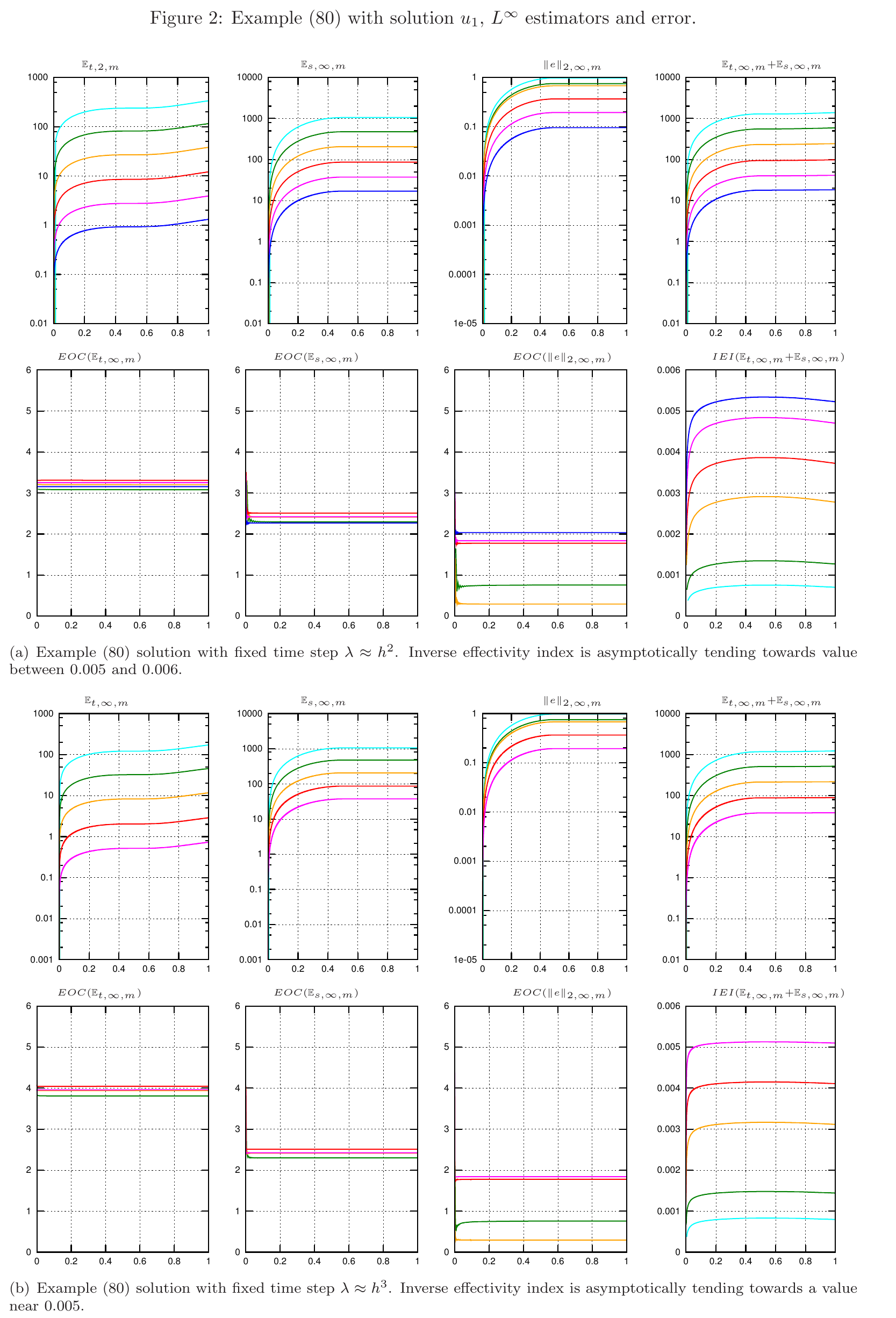}
\end{figure}

\begin{figure}[ht]
		\includegraphics[clip,width=\textwidth]{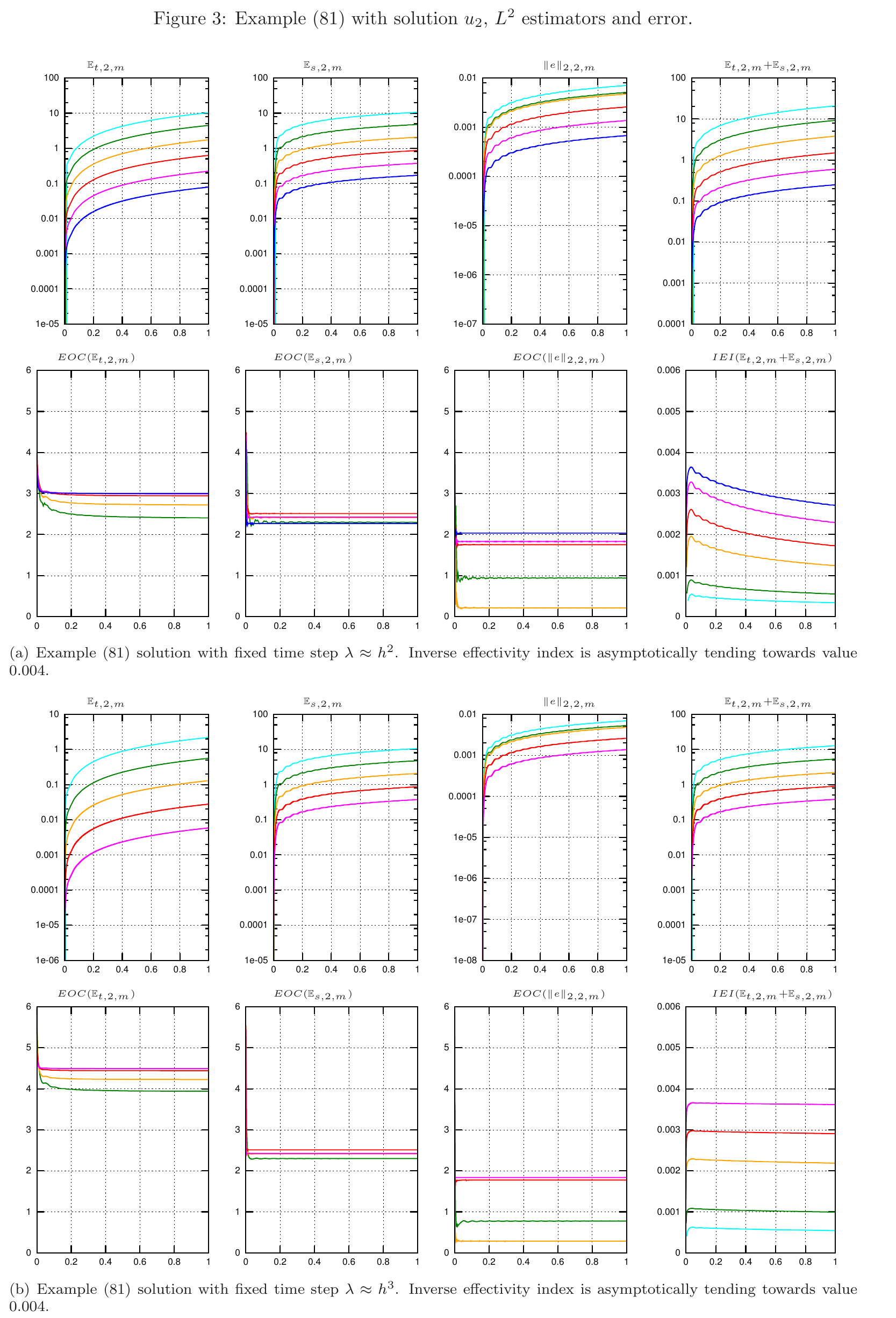}
\end{figure}

\begin{figure}[ht]
		\includegraphics[clip,width=\textwidth]{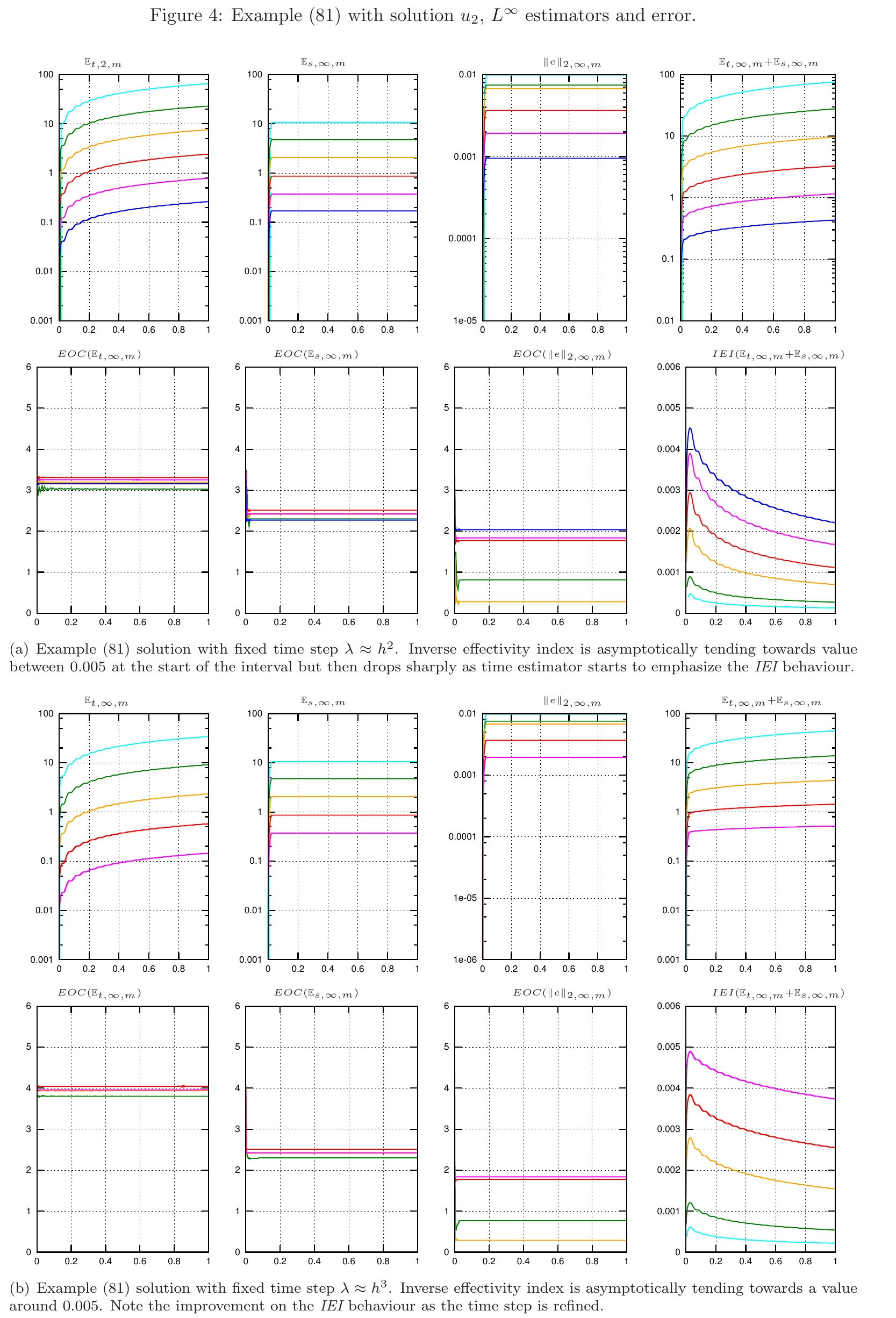}
\end{figure}

\begin{figure}[ht]
		\includegraphics[clip,width=\textwidth]{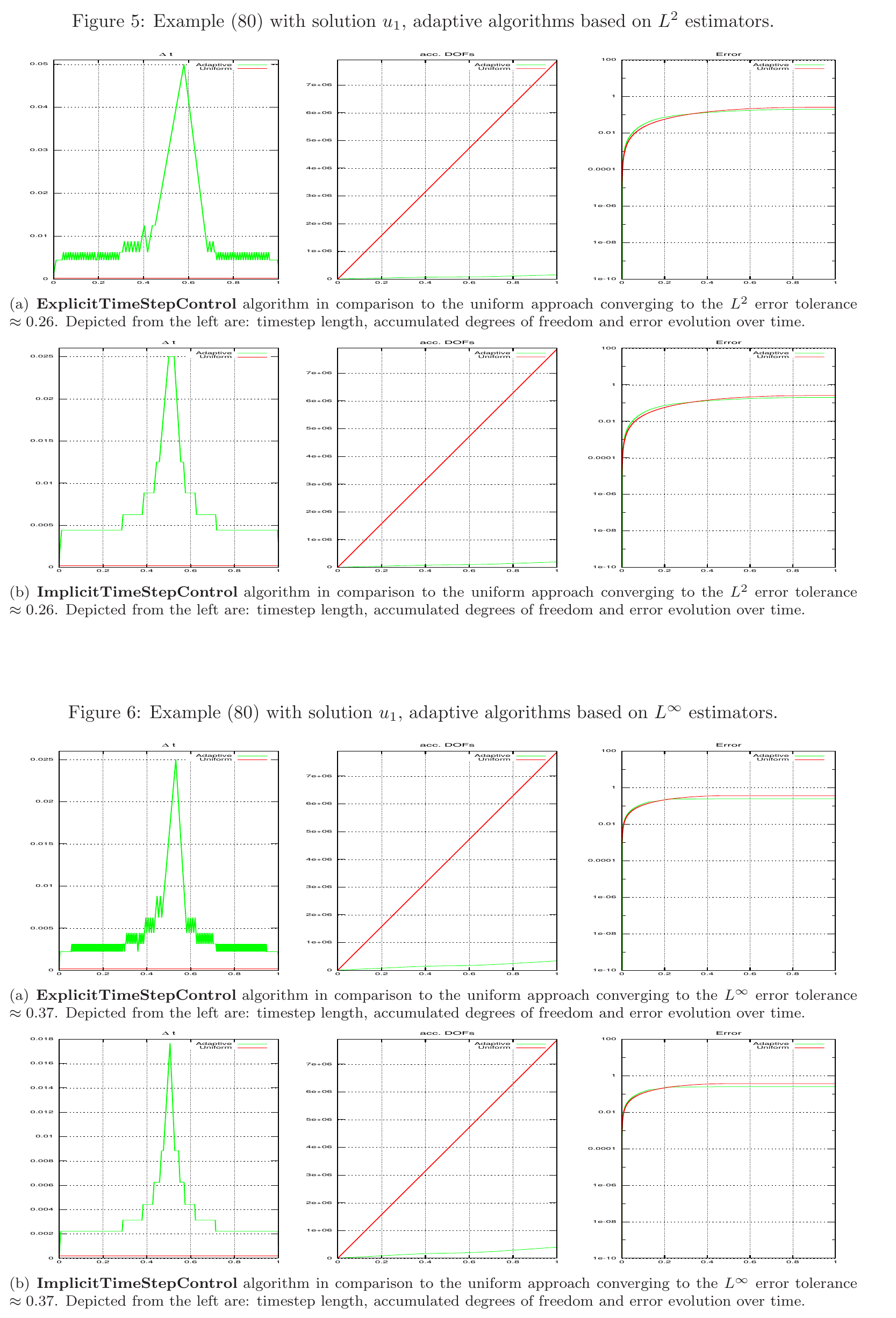}
\end{figure}

\begin{figure}[ht]
		\includegraphics[clip,width=\textwidth]{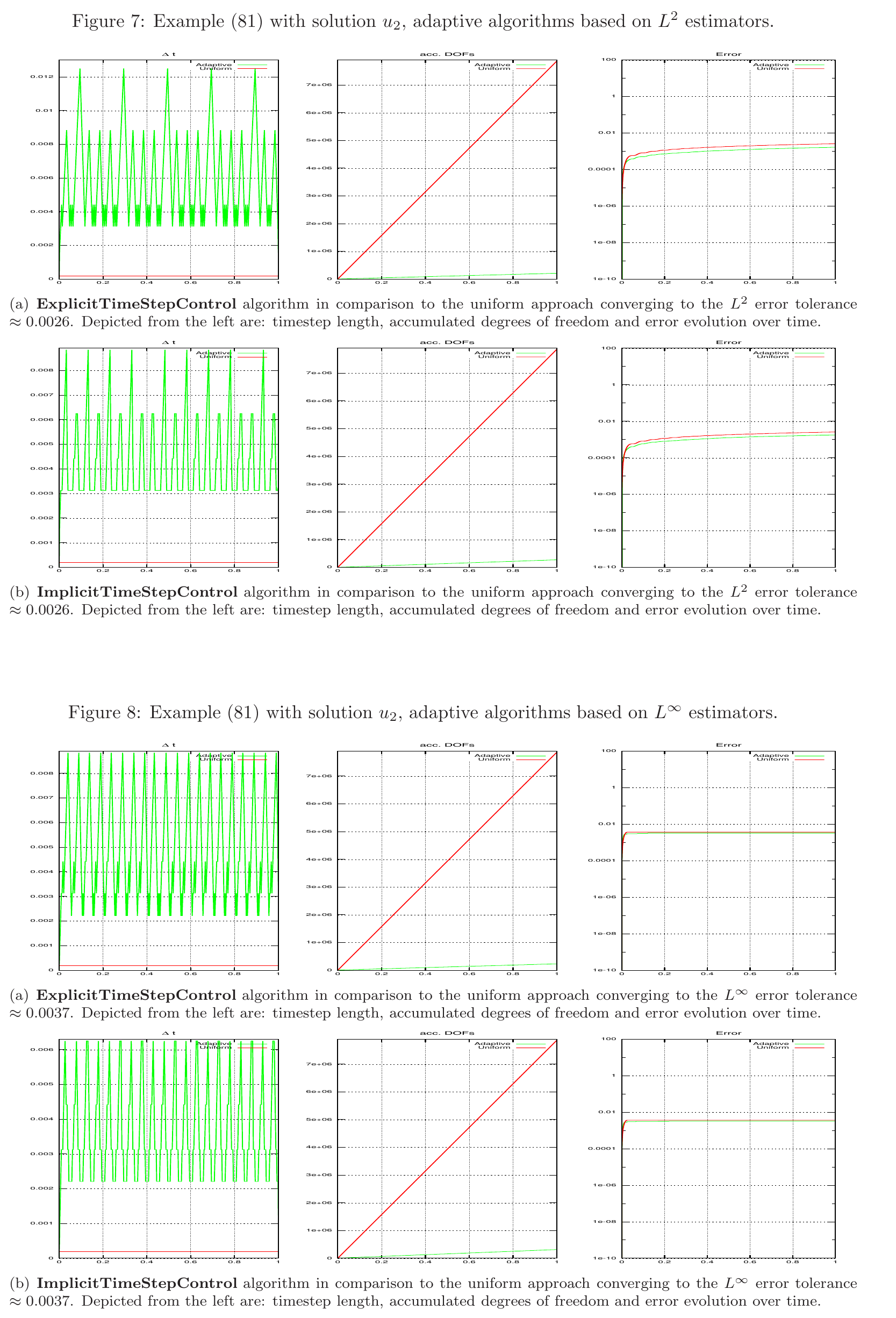}
\end{figure}

\section{Concluding remarks}\label{finalremark}
Residual type a posteriori estimates of errors measured in $L^\infty(L^2)$- and $L^2(L^2)$-norms for a numerical scheme consisting of implicit Euler method in time and discontinuous Galerkin method of local polynomial degrees $r \ge 2$ in space for linear parabolic fourth order problems in space dimensions $2$ and $3$ are presented. Numerical experiments confirming the practical efficiency and reliability of the a posteriori estimators are also presented, along with the use of these a posteriori estimator within adaptive algorithms. It appears that the derived a posteriori bounds and the respective adaptive algorithms can be modified in a straightforward fashion to the original dG method of Baker \cite{baker} and to the $C^0$-interior penalty methods of \cite{hughes, brenner}. Moreover, second order operators can be included in the present analysis, as was done in \cite{thesis_juha}. An extension of these results to nonlinear fourth order parabolic equations remains a future challenge.

\def\cprime{$'$}

\end{document}